\documentclass{article}

\usepackage[utf8]{inputenc}
\usepackage{float}
\usepackage{ stmaryrd }
\usepackage{amsthm}
\usepackage[all]{xy}
\usepackage{amsmath}
\usepackage{amssymb}
\usepackage{amsfonts,mathabx}
\usepackage{yfonts,multirow}
\usepackage{color}
\usepackage{comment}

\usepackage{tikz}
\usetikzlibrary{matrix,arrows,decorations.pathmorphing,patterns}

\pgfdeclaredecoration{complete sines}{initial}
{
    \state{initial}[
        width=+0pt,
        next state=sine,
        persistent precomputation={\pgfmathsetmacro\matchinglength{
            \pgfdecoratedinputsegmentlength / int(\pgfdecoratedinputsegmentlength/\pgfdecorationsegmentlength)}
            \setlength{\pgfdecorationsegmentlength}{\matchinglength pt}
        }] {}
    \state{sine}[width=\pgfdecorationsegmentlength]{
        \pgfpathsine{\pgfpoint{0.25\pgfdecorationsegmentlength}{0.5\pgfdecorationsegmentamplitude}}
        \pgfpathcosine{\pgfpoint{0.25\pgfdecorationsegmentlength}{-0.5\pgfdecorationsegmentamplitude}}
        \pgfpathsine{\pgfpoint{0.25\pgfdecorationsegmentlength}{-0.5\pgfdecorationsegmentamplitude}}
        \pgfpathcosine{\pgfpoint{0.25\pgfdecorationsegmentlength}{0.5\pgfdecorationsegmentamplitude}}
}
    \state{final}{}
}
\usepackage{enumerate}

\newtheorem{theorem}{Theorem}

\newtheorem{proposition}{Proposition}
\newtheorem{lemma}{Lemma}

\newtheorem{corollary}{Corollary}

\newtheorem{fact}{Fact}
\theoremstyle{definition}
\newtheorem{definition}{Definition}

\newcommand{\blue}[1]{\textcolor{blue}{#1}}
\newcommand{\ignore}[1]{}
\newcommand{\Mod}{\mathcal{M}}
\newcommand{\f}{\varphi}
\newcommand{\p}{\psi}

\newcommand\Lprop{\mathcal{L}_{\At}}

\newcommand\M{\mathcal{M}}

\DeclareMathOperator{\At}{Prop}
\DeclareMathOperator{\lit}{Lit}
\newcommand\fvset{{\mathbb{P}_4}}   %set of four-valued assignments
\newcommand\nsset{{\mathbb{P}_{ns}}}    %set of nonstandard assignments
\newcommand\tr{{tr_{ns}^4}}        %translation map from four-valued to nonstandard
\newcommand\trinv{{tr_4^{ns}}}    % translation map from nonstandard to four-valued
\newcommand\Canm{{\mathbb{M}^{can}}}  % Set of canonical models
\newcommand\udb{{\{A,E\}}}                           % updated belief function
\newcommand\lwedge{\sqcap}                     % meet in the information direction
\newcommand\lvee{\sqcup}                         % join in the information direction
\newcommand\BD{{BD_4}}                          % Belnap Dunn square of truth values

\newcommand\swedge{,}   %to shorten conjunction replaced by comma
\newcommand\sneg{\overline} % to shorten negation replaced by overline

\usepackage[round]{natbib}

\title{Probabilities with Gaps and Gluts}
\author{Dominik Klein, Ondrej Majer, Soroush Rafiee Rad}%DK,OM,SRR}
\date{}
\begin{document}
%\textcolor{red}{I am editing now}
\maketitle
\begin{abstract}
Belnap-Dunn logic (BD), sometimes also known as First Degree Entailment, is a four-valued propositional logic that complements the classical truth values of {\it True} and {\it False} with two non-classical truth values {\it Neither} and {\it Both}. The latter two are to account for the possibility of the available information being incomplete or providing contradictory evidence. In this paper, we present a probabilistic extension of BD that permits agents to have probabilistic beliefs about the truth and falsity of a proposition. 
We provide a sound and complete axiomatization for the framework defined and also identify policies for conditionalization and aggregation. Concretely, we introduce four-valued equivalents of Bayes’ and Jeffrey updating and also suggest mechanisms for aggregating information from different sources.\end{abstract}
\section{Introduction}

In learning about a classical system that adheres to the laws of propositional logic, we may be faced with information that does not. Naturally, if information is scarce, our evidence may contain truth value {\it gaps}, neither indicating certain propositions to be true nor  false. But we may also be faced with contradictory information, especially when our insights are gained by combining various bodies of evidence. This may lead to truth value {\it gluts}, i.e. propositions that are labelled  as {\it both} true and false.

There has been many attempts in the literature to develop formal systems for capturing and analyzing such  non-classical situations.  These   are generally divided into two camps. The first is motivated by adopting the philosophical position of dialetheism as defended by \cite{PriestBook, priest2007}, advocating the thesis that there are true contradictions, i.e.\ sentences which are both true and false. Corresponding formal systems should thus allow for assigning both truth values to a sentence simultaneously. Probably the most well known example of such logical systems is the logic LP \citep{priest1979, priest2002}. 

The second camp takes the existence of gaps and gluts as a pathological consequence of imperfect information. Crucially, one may hope  than even  imperfect information would allow for at least some reliable inferences. 
In general, there are two ways to go here. One could either make the set of premises consistent or develop non-trivial inference rules that work on  inconsistent sets of premises. Consistency of premises can be obtained by  focusing on  maximal consistent subsets, cf. \cite{RescherManor,KleinMarra}, or  by employing belief revision, as in AGM systems \citep{agm}.  Mechanisms for dealing with inconsistent information, on the other hand, are developed in a variety of frameworks such as {\em discussive logic} \citep{jaskowski}, {\em adaptive logic}   \citep{batens}, Da Costa's {\em logics of formal inconsistency} (\citeyear{dC1, dC2}), {\em relevant logic} of \cite{anderson} and their variants. 

%The second camp takes the existence of gaps and gluts as a pathological consequence of imperfect information. Crucially, one may still hope  than even such imperfect information would allow for some reliable inference. The approaches that arise in this second camp either perform reasoning with inconsistent premisses by making the set of premisses consistent, (for instance by focusing on the maximal consistent subsets, as Rescher and Manor \cite{RescherManor} do, or by employing belief revision, as in AGM system \cite{agm}) or work with formal systems with mechanisms for dealing with inconsistent information, e.g, {\em discussive logic} \cite{jaskowski}, {\em adaptive logic} introduced by  \cite{batens}, Da Costa's {\em logics of formal inconsistency} \cite{dC1, dC2}, {\em relevant logic} of Anderson \cite{anderson}, and their variants. 

Another well-known logical framework that falls in this last category is Belnap-Dunn logic  \citep[BD. cf.][]{Belnap77, Belnap19, dunn1976}, sometimes also going by the name of First Degree Entailment.  Briefly, this system rests on two assumptions. The first is that gaps and gluts may occur even for boundedly rational agents, as information may be limited (gaps) and the question of whether a given belief set is consistent (i.e. checking for the absence of gluts) is known to be NP-hard.  Building on the latter claim, BD's second assumption is that the logic of information should {\it not} validate the principle of explosion\footnote{The principle of explosion states that {\it every} formula can be derived from a contradiction.}. Just to the contrary, BD stipulates that a body of information may afford us substantial insights about some matter $q$, even if it contains contradictory information about some other $p$ that is completely unrelated to $q$. Belnap-Dunn logic, in short, is a substructural logic, that invalidates explosion and tracks which insights can be inferred from an information base that may contain gaps and gluts. 

But of course, the problem of insufficient or contradictory information does not apply to categorial {\it true-false} information only. Rather, probabilistic information is affected by similar arguments about gaps and gluts as those outlined above. %This is easiest to see within a frequentist setting. There, probability gaps and gluts may occur naturally if probabilities are derived from frequentist sampling of sources that each follow First Degree Entailment. A more direct argument has recently been given within a subjectivist interpretation of probability. 
In his \citeyear{josang} paper, \citeauthor{josang} puts forward a framework for three valued probabilities, incorporating uncertainty as third value that may occur naturally  when evidence is  ambiguous or insufficient. Notably, this framework circumvents the debated principle of insufficient reason by distinguishing situations of insufficient information from those, where equally strong evidence is available for and against some proposition. Later approaches extend this to   four-valued probabilities, where the fourth value represents conflicting information, or gluts. The necessity of gluts is often argued for by considering a Bayesian agent who receives two pieces of mutually contradictory information from sources she judges highly reliable, cf. the firefighter example in \cite{dunn2019contradictory}. 

In short, these arguments call for a four-valued probabilistic generalization of Belnap-Dunn logic in a similar way as classical probability theory  generalizes propositional logic. In a first approach to this project  Michael Dunn  (\citeyear{DunnProb})  has defined a four-valued probabilistic framework and has studied logical properties of the resulting probabilistic entailment. In a similar vein, Childers, Majer and Milne  (\citeyear{childers_at_al}) have put forward a single-valued approach to non-standard probabilities motivated by a frequentist interpretation where probability gaps and gluts may occur naturally if probabilities are derived from sampling  two independent sources. They further substantiate the approach by providing a subjectivist interpretation of non-standard probabilities and the corresponding Dutch Book Argument. 

In the present paper, we offer a novel framework for non-standard probabilities that reconciles \citeauthor{DunnProb}'s and \citeauthor{childers_at_al}'s lines of work (Section \ref{ProbMod}). In doing so, we pursue four major goals. The first is to provide a translation mechanism between four-valued and single-valued non-standard probabilities, showing that these are different but equivalent perspectives on the same phenomenon. (Section \ref{corr}). The second aim relates to an axiomatization of the  system defined. While  \cite{DunnProb} analyzes  logical properties of the probabilistic inference relation ensuing from his approach, no axiomatization of the probabilistic system itself has been put forward so far. To fill this gap, we provide an axiomatization of the non-standard probabilities defined here (Section \ref{ax}) and show this axiomatization sound and complete with respect to a certain class of probabilistic models (Section \ref{soundcomp}). %build upon and further develop Dunn's approach to, in particular, address some major shortcomings of the set up. 
%Our third goal  is to extend and amend Dunn's approach in order to  avoid some conceptual difficulties. 
 While building on Dunn's approach, our framework slightly deviates from  his in order to avoid certain conceptual problems. As we will show, 
%Moreover, we will identify a conceptual shortcoming of previous approaches and develop a framework to avoid these. 
%On a conceptual level, we will show that 
both, J\o sang's three valued probabilities as well as Dunn's four-valued probability implicitly assume all events to be mutually probabilistically independent. Under this assumption, the question of conditionalization trivializes, as no proposition bears any information about any other proposition.  In the present framework, we abandon this independence assumption. Consequentially, the question of conditional probabilities becomes meaningful. Defining and studying an adequate notion of conditionalization is our third goal, pursued in Section \ref{cond}. The fourth goal, finally, is related to aggregation, i.e. the question of how to combine probabilistic information from various sources. Here, we will introduce various policies and study their respective properties (Section \ref{aggr}).
	
\section{Logical Preliminaries}	

We start by giving a brief recollection of  Belnap-Dunn four-valued logic before proceeding to introduce its probabilistic extensions. Belnap-Dunn four-valued logic is defined over a propositional language that is built over  a set $\At$ of propositional variables. Formally, the logical language $\Lprop$   is given by the Backus-Naur form:

$$\varphi:: p\, \vert\, \neg \varphi, \, \vert \, \varphi \wedge \varphi$$

\noindent Disjunction ($\vee$) is defined in the standard way. The main difference to classical propositional logic consists in the way that formulas are evaluated. In  classical propositional logic, evaluations are defined as functions $v: \Lprop \to\{0, 1\}$ that are derived from a valuation on the set of atoms $\At$.  For Belnap-Dunn logic there are two ways to define evaluations. One approach is to define evaluations as functions $v: \Lprop \to  \mathcal{P}(\{0,1\})$. In other words instead of evaluating formulas on the two element lattice 

\begin{center}
 \begin{tikzpicture}[>=stealth',relative]

    \node (U1) at (0,-1) {$\{0\}$};
       \node (U4) at (0,1) {$\{1\}$};

 \path[-,draw] (U1) to (U4);

    \node (L1) at (-0.9,-1.5) {};
     \node (L2) at (0.9,-1.5) {};
      \node (L3) at (-1.5,-0.9) {};
       \node (L4) at (-1.5,0.9) {};
       
       \end{tikzpicture}
\end{center}

they are interpreted on the four element lattice $\BD$

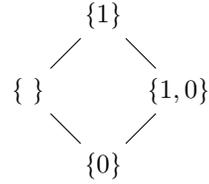
\begin{center}
 \begin{tikzpicture}[>=stealth',relative]

    \node (U1) at (0,-1) {$\{0\}$};
     \node (U2) at (-1,0) {$\{\ \}$};
      \node (U3) at (1,0) {$\{1,0\}$};
       \node (U4) at (0,1) {$\{1\}$};

 \path[-,draw] (U1) to (U2);
 \path[-,draw] (U1) to (U3);
  \path[-,draw] (U2) to (U4);
   \path[-,draw] (U3) to (U4);
    
    \node (L1) at (-0.9,-1.5) {};
     \node (L2) at (0.9,-1.5) {};
      \node (L3) at (-1.5,-0.9) {};
       \node (L4) at (-1.5,0.9) {};
       
       \end{tikzpicture}
\end{center}

\noindent Evaluating formulas  in the four element lattice $\mathcal{P}(\{0,1\})$ allows for the assignment of two new  truth values $\{\}$ and $\{0, 1\}$.  %Representing cases where formulas either obtain neither or both of the classic truth values, these allow to represent so-called truth-value {\bf gaps} and {\bf gluts}. 
These represent so-called truth-value {\bf gaps} and {\bf gluts}, i.e. situations where formulas obtain neither resp. both of the classic truth values. 
Formally, the evaluation is defined inductively, starting from an atomic valuation  $v: \At \to {\mathcal{P}(\{0,1\})}$, by:

	\begin{center}
	\begin{tabular}{lcl}
	$1 \in v(\neg \phi) $&\text{ iff }&$ 0 \in v(\phi)$\\
	$0 \in v(\neg \phi) $&\text{ iff }&$ 1 \in v(\phi)$\\
	$1 \in v(\phi \wedge \psi) $&\text{ iff }&$ 1 \in v(\phi) \text{ and } 1 \in v(\psi)$\\
	$0 \in v(\phi \wedge \psi) $&\text{ iff }&$ 0 \in v(\phi) \text{ or } 0 \in v(\psi)$
	\end{tabular}
	\end{center}

\noindent An alternative approach is to use two separate classical valuations, called  the positive valuation $v^+$ and the negative valuation $v^-$. Building on atomic valuations $v^+: \At \to {\{0,1\}}$ and $v^-: \At \to {\{0,1\}}$, these are defined for  $\phi, \psi \in \Lprop$ as: 

	\begin{center}
	\begin{tabular}{lcl}
	 $v^+(\neg \phi)=1$ &\text{ iff }&$ v^{-}(\phi)=1$\\
	 $v^-(\neg \phi)=1 $ &\text{ iff }&$ v^{+}(\phi)=1$\\
	$v^{+}(\phi \wedge \psi)=1 $ &\text{ iff }&$ v^{+}(\phi)=1 \text{\ and\ } v^{+}(\psi)=1$\\	
	$v^{-}(\phi \wedge \psi)=1 $ &\text{ iff }&$ v^{-}(\phi)=1 \text{\ or\ } v^{-}(\psi)=1$\\	
	\end{tabular}
	\end{center}

Both approaches yield equivalent semantics for Belnap-Dunn logic, as is easily seen.  For  reasons of notational convenience,  we will employ  the double valuation approach. Within this approach, we can  define an entailment relation as $\phi \vDash_{L} \psi$ if and only if $v^{+}(\phi)=1 \Rightarrow v^{+}(\psi)=1$ for all double valuations $(v^+,v^-)$. This entailment relation goes by the name of  {\bf first degree entailment}.

 An important property of  Belnap-Dunn logic, that we will make heavy use of later, is that it admits disjunctive (as well as conjunctive) normal forms. Just as in classical logic, a formula in disjunctive normal form is written as a disjunction of conjunctions of literals. However, unlike in  classical logic, an atom might appear  both positively and negatively within a conjunctive clause.
 
 \begin{theorem}\label{normform}(Theorem 3.9 in \cite{font1997belnap}) Every formula of Belnap-Dunn logic is equivalent to a formula in a conjunctive (disjunctive) normal form.
 \end{theorem}

Moreover, up to permutation of conjuncts and disjuncts, formulas in conjunctive (disjunctive) normal form may be identified with finite families of finite sets of literals %We now show that each formula of BD is equivalent to a unique conjunctively (disjunctively) interpreted irredundant family of sets of literals.
(Theorem 3.15 in \cite{prenosil2018reasoning}).
%A conjunctive (disjunctive)clause is a conjunction (disjunction) of literals. 

%In the above denition we admit empty conjunctions and disjunctions,which are interpreted as t and f, respectively.

\section{Probabilistic Models}\label{ProbMod}

The double valuation approach's starting assumption is that positive and negative evidence are distinct. That is the absence of positive evidence for some $p$ is not the same as negative evidence against $p$ (or positive evidence for $\neg p$, if you will). In particular, there may be gaps, where neither evidence for $p$ nor  against $p$ is available, and gluts, where evidence of both types is present. Within our models, we must hence treat positive and negative evidence separately. In the following, we assume $\At$ finite and constant. Also, we will denote the set of literals over $\At$ by $\lit$, i.e. $\lit:=\At\cup\{\neg p\mid p\in\At\}$.

	\begin{definition}\label{modeldef}
	 A {\bf non-standard model} is a triple  $\mathcal{M} = \langle\Sigma, v^+, v^-\rangle$ where $\Sigma$ is a finite or countably infinite set of states and $v^+, v^-:\Sigma \times \At \to {\{0,1\}}$ are called the positive (negative) valuation function respectively.
	 For $p \in \At$ we let $v^{\pm}(p)= \{s \in \Sigma \, \vert \, v^{\pm}(s, p)=1\}$.
	\end{definition}

Hence, a state $s$ of a model $\M$ might be assigned an inconsistent set of propositions (i.e., $s\in v^+(p)\cap v^-(p) $ for some $p\in Prop$), and may remain undecided about some propositions ($s\not\in v^+(q)\cup v^-(q) $ for some $q\in Prop$). 

Non-standard models provide a semantics for BD. More specifically, logical formulas of $\Lprop$ are evaluated on model-state pairs, using relations $\models^+$ and $\models^-$.  From this, we then obtain the notions of a positive and negative extension.
\begin{definition} Let $\mathcal{M}=\langle\Sigma, v^+, v^-\rangle$  be a non-standard model, $s\in\Sigma$ a state and $\f,\p\in\Lprop$ be formulas. Then\\
$i)$ The {\bf semantics} of $\Lprop$ on $(\mathcal{M},s)$ is given by:
		
	\begin{align*}
	\Mod, s &\models^+ p &&\text{ iff } 	s \in v^{+}(p)\\
	\Mod, s &\models^- p &&\text{ iff } 	s \in v^{-}(p)\\
	\Mod, s &\models^{+} \f \wedge \psi &&\text{ iff } 	\Mod, s \models^{+} \f\text{ and }\Mod, s \models^{+} \psi\\	
	\Mod, s &\models^{-} \f \wedge \psi &&\text{ iff } 	\Mod, s \models^{-} \f\text{ or }\Mod, s \models^{-} \psi\\	
	\Mod, s &\vDash^+ \neg \f &&\text{ iff }  \Mod, s \models^- \f \\
	\Mod, s &\models^- \neg\f  &&\text{ iff } \Mod, s \models^+ \f 
	\end{align*}
	
$ii)$ The {\bf positive} and {\bf negative extensions} of $\f \in\Lprop$ are	
	\begin{align*} |\f|^+_\Mod &= \{s \in \Sigma \,\vert \, \mathcal{M}, s \vDash^+ \f\} \\ |\f|^-_\Mod &= \{s \in \Sigma \,\vert \, \mathcal{M}, s \vDash^- \f\}\ \ \ (= \{s \in \Sigma \,\vert \, \mathcal{M}, s \vDash^+ \neg\f\}  ) \end{align*}
We define the entailment relation between sentences in the usual way: $\phi \vDash^{\pm} \psi$ if and only if for all models $\mathcal{M}$ and states $\sigma$, if $\mathcal{M}, s \vDash^{\pm} \phi$ then $\mathcal{M}, s \vDash^{\pm} \psi$. Observe the obvious connection between positive and negative extension: $|\neg\f|^+_\Mod = |\f|^-_\Mod$.
Moreover, we define the {\bf set of  pure belief,  pure disbelief, conflict} and {\bf uncertainty} about $\f$ as 
\begin{align*}
|\f|_\Mod^b=&|\f|_\Mod^+\setminus|\f|_\Mod^- & |\f|_\Mod^d=&|\f|_\Mod^-\setminus|\f|_\Mod^+\\
|\f|_\Mod^c=&|\f|_\Mod^+\cap|\f|_\Mod^- & |\f|_\Mod^u=&\Sigma\setminus(|\f|_\Mod^+\cup|\f|_\Mod^- ).
\end{align*}
	\end{definition}

\noindent The terms {\it belief} and {\it disbelief}, of course, refer to the intended interpretation as doxastic state. Whenever clear by context, we omit the subscript $\Mod$.

 Towards a semantics of non-standard probability theory, we expand the non-standard model defined above with a probability measure that is classic. Non-classicality of the ensuing probability assignments, then, will be derived from the underlying valuations only, i.e. from the fact that non-standard models allow for  gaps and gluts of truth values.

	\begin{definition}\label{pmodeldef}
		A {\bf probabilistic model}  is a tuple  $\mathcal{M} = \langle\Sigma,\mu, v^+, v^-\rangle$ where $\langle\Sigma, v^+, v^-\rangle$ is a non-standard model and $\mu$ is a probability measure on the full subset algebra of $\Sigma$. 
	\end{definition}
	
\noindent Building on probabilistic models, we can derive two different probability assignments from $\mathcal{M}$, one four-valued, the other single valued. These are:

\begin{definition}
For a probabilistic model $\mathcal{M}=\langle\Sigma, \mu, v^+, v^-\rangle$,\\ 
$i)$  the induced {\bf non-standard probability function} $p_{\mu}:\Lprop\rightarrow\mathbb{R}$ is: $$p_{\mu}(\varphi)=\mu(|\varphi|_\M^+)$$
$ii)$ the induced {\bf four-valued probability function} 	$\hat{p}_{\mu}: \Lprop \to \mathbb{R}^4 $ is 	
	$$\hat{p}_{\mu}(\f) = \left(\mu(|\f|^{b}), \mu(|\f|^{d}), \mu(|\f|^{u}), \mu(|\f|^{c})\right).\footnote{We use the following convention in naming of probability measures i) p-like names stand for syntactic measures (on languages) and $\mu$-like names for measures on spaces. ii) Hat-superscripts are used to denote four-valued probabilities. iii) The Subscript $\mu$ may be used if a syntactic measure is derived from a space.}$$
	\end{definition}

To  end  this section, we'd like to highlight a strong similarity to classic probabilistic models. Classic probability assignments can be derived from  {\it possible worlds models} equipped with a probability function, i.e. finite classical models akin to those in Definition \ref{pmodeldef}. More explicitly, for a classical model of the form $\M=\langle W,v,\mu\rangle$ with $W$ a set of possible worlds, $v\colon W\times\At   \rightarrow {\{0,1\}}$ a valuation, and  $\mu: \mathcal{P}(W)\rightarrow[0;1]$ a probability measure, the probability of some $\varphi$ is given as $\mu([\varphi])$, with $[\varphi]=\{s\in\Sigma\colon\M,s\vDash\varphi\}$. In fact, if $\At$ is finite, every probability assignment to $\Lprop$ can be obtained in this way. 

Moreover, every world $w$ of a possible worlds models $W$ naturally corresponds to  its atomic valuation, which can be represented by the subset $V\subseteq\At$  given by $p\in V$ iff $v(w,p)=1$ for $p\in\At$. In the same vein, each state $\sigma$ of a probabilistic model corresponds to a non-standard possible assignment $n_s \subseteq\mathcal{P}(\lit)$ defined by $p\in n_s$ iff $v^+(w,p)=1$ and $\neg p\in n_s$ iff $v^-(w,p)=1$ for $p\in At$.  
Hence, non standard probabilistic models are obtained from possible world models by replacing classical worlds, i.e. atomic valuations with BD-possible worlds, that is elements of $\mathcal{P}(\lit)$.

%\blue{this last paragraph heavily relies on notation and names that need to be introduced in section 2 (logical structure of...) adapt when available 
%\blue{Alternatively we can define semantics using four truth values corresponding to subsets of $\{T, F\}$ -- EXPAND?? In this case we have a classical definition of satisfaction, but over a non-classical set of truth values. These two approaches correspond to two definitions of non-classical probability we will present. DK. NOT YET DONE. ADD IF DESIRED}

\section{Axioms of Non-standard probability}\label{ax}

In the following, we present a number of axioms for non-standard and  four-valued probabilities. %Being specified fully syntactically, these axioms allow to define non-standard probabilities and four-valued probabilities independent of the probabilistic models of last section or, indeed, of any semantics. However, both sets of axioms presented here are easily seen to be sound w.r.t. to the semantics just presented. That they are also complete will be shown in Section \ref{soundcomp}.
The two sets of axioms given here are easily seen to be sound w.r.t. to the semantics just presented. That they are also complete will be shown in Section \ref{soundcomp}. We can hence use these axioms for a purely syntactic definition of non-standard and four-valued probabilities.

\subsubsection*{Non-standard probabilities}

We begin with axioms for single-valued non-standard probabilities, i.e. probability measures assigning each $\varphi\in\Lprop$ a unique rational number.

\begin{definition} A {\bf non-standard probability assignment} is a function $p\colon\Lprop\rightarrow\mathbb{R}$ satisfying for all $\f,\p\in\Lprop$.
\begin{center}

	\begin{tabular}{lll}
		(A1) &$0 \le p (\f) \le 1$ & (normalization)\\
		(A2)  &if $\f \vDash_{L} \p$ then $p (\f) \le p (\p)$& (monotonicity)\\
		(A3) & $p (\f \wedge \p) + p (\f \vee \p) = p (\f) + p (\p)$.& (import-export rule)	
	\end{tabular}
	\end{center}
	where $\vDash_{L}$ in (A2) is the entailment relation of Belnap-Dunn logic (first-degree entailment).		
	\end{definition}
	
These axioms are strictly weaker than the classic Kolmogorov axioms \citep{kolmogorov}. Axioms (A1)-(A3) can be derived from the Kolmogorov axioms, using that first degree entailment is a sub-relation of classical entailment. In the converse direction, however, only the non-negativity axiom ($p(\varphi)\geq 0$ for all $\varphi$) is derivable from (A1). Neither Kolmogorov's unit axiom $(p(\top) = 1)$ nor the ($\sigma$)-additivity axioms are derivable from (A1)-(A3), as is illustrated by the fact that assigning probability .5 to {\it every} formula satisfies (A1)-(A3). In fact, the import-export axiom is a weak counterpart to additivity, stating that a general rule for adding probabilities that is derivable from the Kolmogorov axioms, $p(\f\vee\p) = p(\f)+p(\p)-p(\f\wedge\p)$, continues to hold. 
Within the above axiomatization, the import-export axioms (A3) is  the only condition regulating the relation between the probability of a formula and its negation. As a result the probabilities of $\f$ and $\neg\f$ need not sum up to 1. The constraint $p(\f\lor\neg \f) + p(\f\land \neg \f) = p(\f) + p(\neg\f)$ allows for probabilistic gaps ($p(\f\lor\neg \f) < 1)$ and gluts ($p(\f\land\neg \f) > 0)$ to occur simultaneously. This squares with our original motivation of establishing  independence between positive and negative evidence. 

%As intended, the probabilities of $\f$ and $\neg\f$ don't sum up to 1 in the non-standard setup. They are bound by the import-export axiom only $p(\f\lor\neg \f) + p(\f\land \neg \f) = p(\f) + p(\neg\f)$, hence the setup allows for probabilistic gaps ($p(\f\lor\neg \f) < 1)$ and gluts ($p(\f\land\neg \f) > 0)$.

\subsubsection*{Four-Valued probabilities}

We now turn to four-valued probability assignments. These are characterized by a total of six axioms. 

\begin{definition} A {\bf four-valued probability assignment} is a function $\hat{p}:\Lprop\rightarrow\mathbb{R}^4$. Writing  $\hat{\f}$ as $(b_\f,d_\f,u_\f,c_\f)$, this function must satisfy
\begin{center}
	\begin{tabular}{ll}
		(D1) &$0 \leq b_\f, d_\f, u_\f, c_\f$ \\ 
		(D2) &$b_\f + d_\f + u_\f + c_\f = 1$ \\
		 (D3)& $b_{\neg\varphi}=d_\varphi$, $c_{\neg\varphi}=c_\varphi$\\
			 (D4)& if $\f \vDash_L \p$ then $b_\varphi+c_\varphi\leq b_\psi+c_\psi$\\
(D5)&   $b_{\varphi\wedge\neg\varphi}=0$, $c_{\varphi\wedge\neg\varphi}=c_\varphi$\\
		 (D6)& $b_\varphi+c_\varphi+b_\psi+c_\psi = b_{\varphi\wedge\psi}+ c_{\varphi\wedge\psi}+ b_{\varphi\vee\psi}+ c_{\varphi\vee\psi}$

	\end{tabular}
	\end{center}
	where $\vDash_{L}$ is first-degree entailment and $\f,\p\in\Lprop$		
	\end{definition}

The four entries of $\hat{p}$ stand for pure belief (i.e. $\varphi$ is true and $\neg\varphi$ is not), pure disbelief, uncertainty and conflict respectively. 	
Let us briefly explain the axioms. The first two axioms (D1) and (D2) are classicality axioms, stating that  probabilities are non-negative and that the probabilistic masses of pure belief, pure disbelief, conflict and uncertainty must add up to 1. This  reflects the intuition that the four cases are mutually exclusive and jointly exhaustive, i.e. that the metatheory of gaps and gluts is classical. 

Axioms (D3)-(D6) then represent structural relations between the four-valued assignments. (D3) emphasizes the strong relation between $\f$ and $\neg \f$: belief in one is the same as disbelief in the other, while both share the same conflict and uncertainty. (D4) is a direct counterpart of axioms (A2) above, stating that the total belief in $\varphi$ (i.e. the sum of pure belief in $\varphi$ and belief in $\varphi$ and $\neg\varphi$ together) must be monotonous under first degree entailment.  
(D5) expresses that an agent cannot have pure belief in contradictory formulas of the form $\varphi\wedge\neg\varphi$. A fortiori, the conflict about $\varphi\wedge\neg\varphi$ must be derived from (and equal to) conflict about $\varphi$ alone. 
(D6), finally, is a counterpart to the import-export axiom (A3). Briefly, it states the total beliefs (i.e. the sum of pure belief and conflict together) of $\varphi,\psi,\varphi\vee\psi$ and $\varphi\wedge\psi$ must satisfy the import-export rule. 

We should note that the axioms presented here are weaker than those put forward in \cite{DunnProb}. There, the probability of a conjunction $\f\wedge\p$ is determined by its conjuncts through: 
\begin{center}
\begin{tabular}{ll}
$b_{\f\wedge\p}=b_\f\cdot b_\p$&$d_{\f\wedge\p}=d_\f+d_\p-d_\f d_\p+c_\f u_\p+u_\f c_\p$\\
$u_{\f\wedge\p}=u_\f b_\p+b_\f u_\p+u_\f u_\p$&$c_{\f\wedge\p}=b_\f c_\p+c_\f b_\p+c_\f c_\p$
\end{tabular}
\end{center}
\noindent A similar axiom for three valued probabilities (true/false/uncertain) can be found in \cite{josang}. Notably, such definition makes conjunctions truth functional, i.e. the probability of  $\f\wedge\p$ is fully  determined by the probabilities of $\f$ and $\p$. We take this to be too strong, especially given that no such functional dependence holds  in classic probability theory. Moreover  this truth functional approach implies that all propositions are mutually probabilistically independent - precluding any interesting notions of conditionalization. To see this, assume that $\f$ and $\p$ are classical, i.e. $\hat{p}(\f)=(b_\f,d_\f,0,0)$ and  $\hat{p}(\p)=(b_\p,d_\p,0,0)$. Then the above definition simplifies to $\hat{p}(\f\wedge\p)=(b_\f b_\p,d_\f+d_\p-d_\f d_\p,0,0)$. With other words, the probability (belief) in $\f\wedge\p$ is the product of the probabilities of $\f$ and $\p$ - which exactly is the definition of probabilistic independence.

In the following section we will show a strong correspondence between non-standard and four-valued probability assignments. Thereafter, we show axiom systems (A1)-(A3) and (D1)-(D6) to be sound and complete  with respect to the class of probabilistic models defined above (Section \ref{soundcomp}). In  Section \ref{cond} we then discuss approaches to conditionalization in either setting.

\section{Correspondence between non-standard and four-valued probabilities}\label{corr}

We have so far presented two different frameworks for non-standard probability, one real-valued, the other with values in $\mathbb{R}^4$. As we show now, both are different but equivalent perspectives on the same phenomenon. To this end, let $\nsset$ and $\fvset$ be the set of non-standard and four-valued probability assignments respectively. That is, $\nsset$ is the set of functions $\Lprop\rightarrow\mathbb{R}$ satisfying (A1)-(A3) while  $\fvset$  consists of all mappings $\Lprop\rightarrow\mathbb{R}^4$ satisfying (D1)-(D6). We will show the translation map  $\tr\colon\fvset\rightarrow\nsset$ defined by
\begin{align*}
 \tr(\hat{p})(\varphi):=b_\varphi+c_\varphi\hspace{1cm}\text{ where }\hat{p}(\varphi)=(b_\f,d_\f,u_\f,c_\f)
\end{align*}
to be a bijection. In the opposite direction, the map $\trinv\colon\nsset\rightarrow\fvset$ is given by 

\begin{align*}
\trinv(p)(\varphi):= &(p(\varphi) -  p(\varphi\wedge\neg\varphi),\ p(\neg\varphi) -  p(\varphi\wedge\neg\varphi),\\&\textcolor{white}{(}1-p(\varphi)-p(\neg\varphi)+p(\varphi\wedge\neg\varphi),\ p(\varphi\wedge\neg\varphi))
\end{align*}

\noindent As expected, the maps $\tr$ and $\trinv$ are inverse to each other:
\begin{theorem}\label{translthm}
$\tr$ and $\trinv$ are well-defined. Moreover  $\trinv\circ\tr = id_{\fvset}$ and $\tr\circ\trinv=id_{\nsset}$
\end{theorem}

\noindent Moreover,  the translation maps $\tr$ and $\trinv$ cohere with the way we defined non-standard and four-valued assignments on a given probabilistic model. 
\begin{theorem}\label{commtriang}
Let $\mathcal{M}=\langle\Sigma, \mu, v^+, v^-\rangle$ be a probabilistic model and $p_\mu$ and $\hat{p}_\mu$ the induced non-standard and four-valued probability functions. Then $\tr\circ\hat{p}_\mu=p_\mu$ and  $\trinv\circ p_\mu=\hat{p}_\mu$.
\end{theorem}

\noindent The remainder of this section is devoted to showing these two results. 

\begin{proof}[Proof of Theorem \ref{translthm}]
To see that $\tr$ is well defined let $p'= \tr(\hat{p})$ for a fixed $\hat{p}\in\fvset$. First,  note that for any $\p\in\Lprop$ with $\hat{p}(\p)=(b_\p,d_\p,u_\p,c_\p)$ we have $0\leq b_\p+c_\p\leq 1$, showing that $p'$ satisfies (A1). To see that $p'$ satisfies (A2) assume that $\varphi\vDash_L\psi$. By   (D4), we have that $b_\varphi+c_\varphi\leq b_\psi+c_\psi$ and hence $p'(\varphi)\leq p' (\psi)$. For (A3) finally, note that by (D6) we have  for any $\varphi,\psi\in \Lprop$ that $b_\varphi+c_\varphi+b_\psi+c_\psi = b_{\varphi\wedge\psi}+ c_{\varphi\wedge\psi}+ b_{\varphi\vee\psi}+ c_{\varphi\vee\psi}$ which immediately implies that $p'(\varphi) +p'(\psi) = p'(\varphi\wedge\psi) + p'(\varphi\vee\psi)$.

Next, we show that also $\trinv$ is well defined. For this fix $p\in\nsset$. For $\p\in\Lprop$ denote $\tr(p)(\p)$ by $(b_\p,d_\p,u_\p,c_\p)$. Using this notation, we obtain
\begin{align*}
b_\p+d_\p+u_\p+c_\p=
&\textcolor{white}{+}p(\p) -  p(\p\wedge\neg\p)+p(\neg\p) -  p(\p\wedge\neg\p)\\&+1-p(\p)-p(\neg\p)+p(\p\wedge\neg\p)+p(\p\wedge\neg\p)
\end{align*}
the latter term is easily seen to equal 1, showing (D2). For (D1) note that $\p\wedge\neg\p\vDash_L \p$ and $\p\wedge\neg \p\vDash_L \neg \p$. By (A2), we have that $p(\p\wedge\neg \p)\leq p(\p),p(\neg\p)$ which, together with (A1) implies that $b_\p,d_\p,c_\p\geq 0$. Finally, by (A1) and (A3), $$1-  p(\p)-p(\neg\p)+p(\p\wedge\neg\p) \geq p(\p\vee\neg\p) -  p(\p)-p(\neg\p)+p(\p\wedge\neg\p)=0$$ and hence $u_\p\geq 0$. The first half of (D3) follows from the fact that $b_\p=p(\p) -  p(\p\wedge\neg\p)=d_{\neg\p}$, using that $\p\Dashv\vDash_L\neg\neg\p$ and hence, by (A2), $p(\p)=p(\neg\neg\p)$. The second half follows from the fact that $\p\wedge\neg\p\Dashv\vDash_L\neg\p\wedge\neg\neg\p$ and hence, by (A2),  $p(\p\wedge\neg\p)=p(\neg\p\wedge\neg\neg\p)$. 
Similarly, (D4) can be derived from (A2) together with the fact that $b_\p+c_\p=p(\p) -  p(\p\wedge\neg\p)+ p(\p\wedge\neg\p)=p(\p)$. Using the latter fact again, (D6) is an immediate consequence of (A3). For (D5), finally, note that $\p\wedge\neg \p \Dashv\vDash_L \p\wedge\neg \p\wedge\neg(\p\wedge\neg \p)$ and hence, by (A2), $p(\p\wedge\neg\p)=p(\p\wedge\neg \p\wedge\neg(\p\wedge\neg \p))$. This implies that $c_{\p\wedge\neg \p}=c_\p$ and that  $b_{\p\wedge\neg\p}=p(\p\wedge\neg\p)-p(\p\wedge\neg \p\wedge\neg(\p\wedge \neg\p))=0$.

Finally, we show that $\trinv\circ\tr = id_{\fvset}$ and $\tr\circ\trinv=id_{\nsset}$, i.e. that $\trinv$ and $\tr$ are left and right inverses of each other. We begin by showing that $\tr(\trinv(p))=p$ for any $p\in\nsset$. For $\varphi\in \Lprop$, we have that $\trinv(p)(\varphi)$ equals $$  \left(p(\varphi) -  p(\varphi\wedge\neg\varphi),p(\neg\varphi) -  p(\varphi\wedge\neg\varphi),1-p(\varphi)-p(\neg\varphi)+p(\varphi\wedge\neg\varphi),p(\varphi\wedge\neg\varphi)\right).$$
Hence $\tr(\trinv(p))(\varphi)=p(\varphi) -  p(\varphi\wedge\neg\varphi)+p(\varphi\wedge\neg\varphi)=p(\varphi)$ as desired.

For the converse direction, let $\hat{p}\in\fvset$. We have to show that $\trinv(\tr(\hat{p}))=\hat{p}$. For this, let $\varphi\in \Lprop$ and denote $\hat{p}(\p)$ by $(b_\p,d_\p,c_\p,u_\p)$ for any $\p\in\Lprop$. By axioms (D3) and (D5) we have that $b_{\neg\varphi}=d_\varphi$, $c_{\neg\varphi}=c_\varphi$, $b_{\varphi\wedge\neg\varphi}=0$ and $c_{\varphi\wedge\neg\varphi}=c_\varphi$. Hence, the values of $\tr(\hat{p})(\varphi),\ \tr(\hat{p})(\neg\varphi)$ and $\tr(\hat{p})(\varphi\wedge\neg\varphi)$ are $b_\varphi+c_\varphi,\ d_\varphi+c_\varphi$ and $c_\varphi$ respectively. We then get that 
 \begin{align*}
&\trinv(\tr(\hat{p}))(\varphi)\\ =& (\tr(\hat{p})(\varphi) -  \tr(\hat{p})(\varphi\wedge\neg\varphi),\tr(\hat{p})(\neg\varphi) -  \tr(\hat{p})(\varphi\wedge\neg\varphi),\\&\phantom{(}\tr(\hat{p})(\varphi\wedge\neg\varphi),1-\tr(\hat{p})(\varphi)-\tr(\hat{p})(\neg\varphi)+\tr(\hat{p})(\varphi\wedge\neg\varphi))\\
=&(b_\varphi+c_\varphi-c_\varphi,d_\varphi+c_\varphi-c_\varphi,c_\varphi,1-(b_\varphi+c_\varphi)-(d_\varphi+c_\varphi)+c_\varphi)\\
=&(b_\varphi,d_\varphi,c_\varphi,1-b_\varphi-d_\varphi-c_\varphi)=(b_\varphi,d_\varphi,c_\varphi,u_\varphi)
\end{align*}
where the last equation employs  (D2). Hence $\tr(\trinv(\hat{p}))=\hat{p}$ as desired.
\end{proof}

\begin{proof}[Proof of Theorem \ref{commtriang}]

For $\varphi\in\Lprop$ denote $\hat{p}_\mu(\f)$ by $(b_\f,d_\f,u_\f,c_\f)$. By Definition \ref{pmodeldef}, we have 
\begin{align*}
b_\f&=\mu(|\f|_\Mod^b)=\mu(|\f|_\Mod^+\setminus|\f|_\Mod^-)\\ 
c_\f&=\mu(|\f|_\Mod^c)=\mu(|\f|_\Mod^+\cap|\f|_\Mod^- )
\end{align*}
Hence, \begin{align*}\tr(\hat{p}_\mu)(\f)&=b_\f+c_\f=\mu(|\f|_\Mod^b)+\mu(|\f|_\Mod^c)\\&=\mu(|\f|_\Mod^+\setminus|\f|_\Mod^-)+\mu(|\f|_\Mod^+\cap|\f|_\Mod^- )=\mu(|\f|_\Mod^+)\end{align*} By definition, the latter term is exactly $p_\mu(\varphi)$. Thus $\tr\circ\hat{p}_\mu=p_\mu$, as desired. 

Moreover, the latter formula implies that  $\trinv\circ\tr\circ\hat{p}_\mu=\trinv\circ p_\mu$. By Theorem \ref{translthm}, we have have $\trinv\circ\tr=id_{\fvset}$. Hence, the last equation reduces to  $\hat{p}_\mu=\trinv\circ p_\mu$, proving the second part of the theorem. 
\end{proof}

\section{A Completeness Result.}\label{soundcomp}

Having shown that non-standard and four-valued probability assignments are equivalent, as witnessed by the bijection $\tr\colon\fvset\rightarrow\nsset$, we now turn our attention to the class of probability functions that are induced by probabilistic models. As it turns out, these are fully characterized by our axioms (A1)-(A3). More specifically, we will show that axioms (A1)-(A3) are a sound and complete characterization of the induced non-standard probability functions of probabilistic models. Of course, by Theorems \ref{translthm} and \ref{commtriang}, this implies that also (D1)-(D6) are a sound and complete characterization of the induced four-valued probability functions of probabilistic models. In fact, the soundness part is easy to check:

\begin{lemma}\label{sclem}
Let  $\mathcal{M}=\langle\Sigma, \mu, v^+, v^-\rangle$ be a probabilistic model and $p_\mu$ the induced non-standard probability function. Then $p_\mu$ satisfies (A1)-(A3).
\end{lemma}

Towards completeness, we will show a stronger result. Recall that {\it completeness} expresses that every $p\in\nsset$ is the induced non-standard probability function of {\it some} probabilistic model $\Mod$. This $\Mod$ may, however, not be unique as $p$ may be not expressive enough to completely determine all properties of $\Mod$. As we will show $\Mod$, is almost unique. More specifically, we determine a class $\Canm$ of canonical models such that every $p\in\nsset$ is the induced non-standard probability function of {\it exactly one} $\Mod\in\Canm$. 

\begin{definition}
$i)$ We call a probabilistic model $\mathcal{M}=\langle\Sigma, \mu, v^+, v^-\rangle$ {\bf canonical} iff $\Sigma=\mathcal{P}(\lit)$ and $v^+,v^-$ satisfy
\begin{align*}
v^+(p)=\{\sigma\in\mathcal{P}(\lit)\mid p\in\sigma\}\hspace{1cm}v^-(p)=\{\sigma\in\mathcal{P}(\lit)\mid \neg p\in\sigma\}
\end{align*}
$ii)$ $\mathbf{\Canm}$ is the set of canonical probabilistic models. 
\end{definition}

{\it Remark: }
The set $\Canm$ is representative of the set of all models in the following sense: For any probabilistic model $\M=\langle\Sigma, \mu, v^+, v^-\rangle$, there is a unique canonical model ${\M}_c=\langle\mathcal{P}(\lit), \mu_c, v^+_c, v^-_c\rangle$ and a unique function $f\colon\M\rightarrow{\M}_c$ such that $x\in v^{\pm}(p)\Leftrightarrow f(x)\in v_{c}^{\pm}(p)$  and $\mu_c(\sigma_c)=\mu(f^{-1}(\sigma_c))$ for all $\sigma_c\in\mathcal{P}(\lit)$. In particular, $p_\mu(\f)=p_{\mu_c}(\f)$ for all $\f\in\Lprop$. The main theorem of this section is:
%\end{remark}

\begin{theorem}\label{repthm}
For any $p\in\nsset$ there is a unique canonical model $\mathcal{M}_p=\langle\mathcal{P}(\lit), \mu, v^+, v^-\rangle$ with induced non-standard probability function $p_\mu$ such that $p=p_\mu$.
\end{theorem}

\begin{corollary}\label{sccor}
Axioms (A1)-(A3) are sound and complete with respect to the class of induced non-standard probability functions of probabilistic models. 
\end{corollary}

By Theorems \ref{translthm} and \ref{commtriang}, the previous result readily translates to the level of four-valued probability functions. 

\begin{theorem}\label{repthmfour}
For any $\hat{p}\in\fvset$ there is a unique canonical model $\mathcal{M}_{\hat{p}}=\langle\mathcal{P}(\lit), \mu, v^+, v^-\rangle$ with induced four-valued probability function $\hat{p}_\mu$ such that $\hat{p}=\hat{p}_\mu$.
\end{theorem}

\begin{corollary}
Axioms (D1)-(D6) are sound and complete with respect to the class of induced four-valued probability functions of probabilistic models. 
\end{corollary}

\begin{proof}[Proof of Theorem \ref{repthm}]
Fix ${p}\in\nsset$. Let $\Sigma=\mathcal{P}(\lit)$ and let $v^\pm\colon\At \to \mathcal{P}(\Sigma)$  be defined as $v^+(q)=\{\sigma\in\Sigma\mid q\in\sigma\}$ and $v^-(q)=\{\sigma\in\Sigma\mid \neg q\in\sigma\}$ respectively. We will construct a classic probability  function $\mu\colon\mathcal{P}({\Sigma})\rightarrow[0;1]$ such that the canonical model $\mathcal{M}=\langle\Sigma, \mu, v^+, v^-\rangle$ satisfies $p_\mu=p$. It suffices to construct the underlying probability mass function $W \colon \Sigma\rightarrow[0;1]$, i.e. the function satisfying $W(x)= \mu( \{x \})$ for $x \in \Sigma$. 
We will do so by induction on $|x|$ for $x\in  \Sigma=\mathcal{P}(\lit)$. The construction proceeds in three steps. As an induction base, we set $\mu(x_{max})$ with $x_{max}$ the unique element in $\Sigma$ with $|x_{max}| =  |\lit|$. In the induction step, we define  $\mu(x)$ for all $x$ with $|x|=k\geq 1$, assuming that $\mu(y)$ has already been defined  for all $y$ with $|y|>k$. In the last step, finally, we define $\mu(\emptyset)$, where $\emptyset$ is the unique element of $\Sigma$ of cardinality 0.

We will need to ensure that that   $\mu([\varphi])= p(\varphi)$ for all $\varphi\in\Lprop$, where $[\varphi]$ denotes the truth set of $\varphi$ in the non-standard model $\langle\Sigma,v^+,v^-\rangle$,  i.e.  $[\varphi]=\{x\subseteq\lit\mid\bigwedge_{q\in x}q\vDash_L\varphi\}$.  Note that by the normal form theorem (Theorem \ref{normform}) and axiom (A2), it suffices to show this property for all $\varphi\in\Lprop$ that are in disjunctive normal form. Moreover note that for any  $\varphi,\psi\in \Lprop$ in disjunctive normal form, we have that $\mu([\varphi\vee\psi])=\mu([\varphi])+\mu([\psi])-\mu([\varphi\wedge\psi])$, as witnessed by \begin{align*}\mu([\varphi\vee\psi])=&\sum_{x\vDash\varphi\vee\psi}\mu(x)=\sum_{x\vDash\varphi}\mu(x)+\sum_{x\vDash\psi}\mu(x)-\sum_{x\vDash\varphi\wedge\psi}\mu(x)\\=&\mu([\varphi])+\mu([\psi])-\mu([\varphi\wedge\psi]).\end{align*} By (A3), hence, knowing that $\mu([*])= p(*)$ for $*\in\{\varphi,\psi,\varphi\wedge\psi\}$ guarantees that $\mu([\varphi\vee\psi])= p(\varphi\vee\psi)$. It thus suffices to show that $\mu([\varphi])= p(\varphi)$ whenever $\varphi$ is a conjunction of literals, i.e. of the form $\bigwedge_{q\in x}q$ with $x\subseteq\lit$. We will show this property to hold alongside our inductive construction.

For the first step,   let $x_{max}$ be $\bigwedge_{q\in\lit} q$,  the unique element in $ \mathcal{P}(\Sigma)$ of maximal cardinality. Note that $\{x_{max}\}$ is the truth set of the formula $\bigwedge_{q\in\lit} q$. We thus set $ W (x_{max}):= p(\bigwedge_{q\in\lit} q)$. By axiom (A1) we have that $0\leq W (x_{max})\leq1$.

For the inductive step let $k\geq 1$ and assume that $ W (y)$ has already been defined for all $y$ with $|y|>k$. We simultaneously define $ W (x)$ for all $x\subseteq\lit$ with $|x|=k$. Let such $x$ be given. Note that the truth set of $\bigwedge_{q\in x}q$ is $\{y\subseteq\lit\mid x\subseteq y\}$. By induction assumption $ W (y)$ is already defined for all $\{y\subseteq\lit\mid x\subset y\}$. We can hence define \begin{align}\label{eqn1} W (x):= p(\bigwedge_{q\in x}q)-\sum_{\{y\subseteq\lit\mid x\subset y\}} W (y).\end{align}
We have that $ W (x)\leq  p(\bigwedge_{q\in x}q)$ and thus $ W (x)\leq 1$. On the other hand, note that $\{y\subseteq\lit\mid x\subset y\}$ is the truth set of $\bigvee_{y\supset x}\bigwedge_{q\in y}q$. Hence, by induction assumption, \begin{align}\label{eqn2}\sum_{\{y\subseteq\lit\mid x\subset y\}} W (y)= p(\bigvee_{y\supset x}\bigwedge_{q\in y}q).\end{align} 

Moreover, note that $\bigvee_{y\supset x}\bigwedge_{q\in y}q\vDash\bigwedge_{q\in x}q$ and hence, by (A2),  $p(\bigvee_{y\supset x}\bigwedge_{q\in y}q)\leq p(\bigwedge_{q\in x}q)$. Combining this inequality with (\ref{eqn1}) and (\ref{eqn2}) yields $ W (x)\geq 0$. 

For the last step, finally, assume that $ W (x)$ is already defined for all $x\neq\emptyset$. We then set $ W (\emptyset) = 1-\sum_{x\neq\emptyset} W (x)$. It follows immediately that $\sum_{x\in\Sigma} W (x)=1$. Moreover, by our induction, $ W (x)\geq 0$ for all $x\neq\emptyset$, hence $ W (\emptyset)\leq 1$. On the other hand, note that  $\{x\subseteq\lit\mid x\neq\emptyset\}$ is the truth set of $\bigvee_{q\in\lit}q$. By induction assumption, $\sum_{x\neq\emptyset} W (x)= p(\bigvee_{q\in\lit}q)$. By axiom (A1), hence, $ W (\emptyset)\geq 0$. 

Along the lines of the proof, we have ensured that $\mu([\varphi])= p(\varphi)$ for all $\varphi$ of the form $\bigwedge_{q\in x}q$ for some $x\subseteq\lit$. By the above remark, this ensures that $\mu([\varphi])= p(\varphi)$ for all $\varphi$, i.e. that $p_\mu=p$. 
\end{proof}

To end the static parts of this paper, we provide a graphical overview over the relationships identified so far. By  Theorems \ref{translthm} to \ref{repthmfour}, the diagram in Figure \ref{RelPic} commutes. Moreover, each pair of opposite arrows in the upper half of the diagram, i.e. the pairs $(\tr,\trinv)$, $(p\rightarrow\mathcal{M}_p,\mu\rightarrow p_\mu)$ and $(\hat{p}\rightarrow\mathcal{M}_{\hat{p}},\mu\rightarrow \hat{p}_\mu)$ are left- and right inverses to each other. 
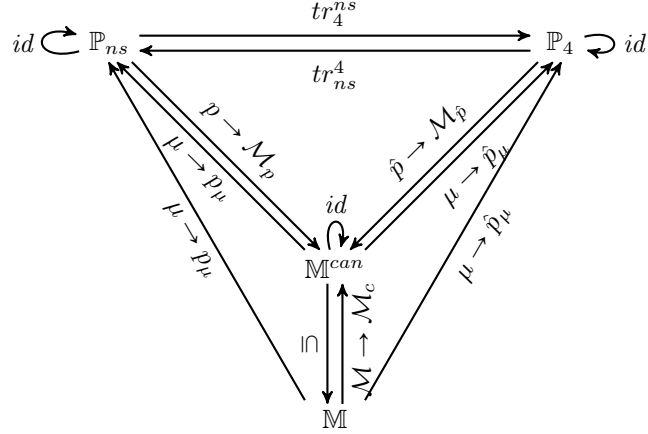
\begin{figure}
\begin{center}
\begin{tikzpicture}[->,>=stealth',auto,node distance=3cm,
  thick,main node/.style={circle,draw,font=\sffamily\Large\bfseries}]

  \node(1) at (0,0){$\nsset$};
  \node  (2) at (6,0) {$\fvset$};
  \node  (3) at (3,-3){$\Canm$};
  \node (4) at (3,-5){$\mathbb{M}$};
 
  \draw   (0.4,0.1)  to  node[above] {$\trinv$ } (5.6,0.1);
   \draw   (5.6,-0.1)  to node[below] {$\tr$ }  (0.4,-0.1);
     \draw   (0.3,-0.25) to node[above,sloped]{$p\rightarrow\mathcal{M}_p$} (2.8,-2.75);
   \draw   (2.6,-2.75)  to  node[below,sloped]{$\mu\rightarrow p_\mu$}(0.1,-0.25);
 \draw   (5.7,-0.25) to node[above,sloped]{$\hat{p}\rightarrow\mathcal{M}_{\hat{p}}$} (3.2,-2.75);
   \draw   (3.4,-2.75)  to  node[below,sloped]{$\mu\rightarrow \hat{p}_\mu$}(5.9,-0.25);
 \draw   (2.9,-3.2)  to  node[below,sloped] {$\subseteq$ } (2.9,-4.8);
   \draw   (3.1,-4.8)  to node[below,sloped] {$\ \mathcal{M}\rightarrow\mathcal{M}_c$}  (3.1,-3.2);
   \draw   (2.6,-4.75)  to  node[below,sloped]{$\mu\rightarrow p_\mu$}(-0.01,-0.25);
   \draw   (3.4,-4.75)  to  node[below,sloped]{$\mu\rightarrow \hat{p}_\mu$}(6.01,-0.25);
   \path[thick,-to]  (1) edge [loop left]  node[left] {{\it id}} (1);
   \path[thick,-to]  (2) edge [loop right]  node[right] {{\it id}} (2);
   \path[thick,-to]  (3) edge [loop above]  node[above] {{\it id}} (3);
 \end{tikzpicture}

\end{center}
\caption{The relationships identified so far. By Theorems \ref{translthm} to \ref{repthmfour}, this diagram commutes.}\label{RelPic}
\end{figure}

\section{Conditioning}\label{cond}

In a classic setting, Bayesian conditioning on a formula $\f$ describes a situation, where $\f$ is learned to be true with probability 1 -- and hence $\neg \f$ true with probability 0.  A generalization of this rule is Jeffrey conditioning, where an agent may learn the probability of $\f$ to be any value in $q\in[0;1]$, rather than only the extremal value of 1 (or 0, when $\neg \f$ is learned) permitted in Bayes' conditioning. 
%Bayesian conditioning, where one learns a formula $\f$ with extremal values (1 or, 0 when $\neg \f$ is learnt), is a special case of Jeffrey conditioning, where an agent may learn the probability of $\f$ to be any $q\in[0;1]$. 

Either method is best illustrated semantically. Within a classical setting, any formula $\f$ defines a binary partition $\{[\f],[\neg\f]\}$ on the state space, cf. Figure \ref{classicfig}. Jeffrey conditioning  is then executed by linearly expanding or contracting the original measure $\mu$ on $[\f]$ and $[\neg\f]$ to some new $\overline{\mu}$ in such a way that $\overline{\mu}([\f])=q$ and $\overline{\mu}([\neg\f])=1-q$. We hence get for any $\f\in\Lprop$ that
\begin{align}
\overline{\mu}([\p])=\mu([\p\wedge\f])\frac{q}{\mu([\f])}+\mu([\p\wedge\neg\f])\frac{1-q}{\mu([\neg\f])}
\label{stJeffrey}
\end{align}
which, in the case of Bayesian condition (i.e. $q=1$) reduces to the well-known formula $\overline{\mu}([\p])=\frac{\mu([\p\wedge\f])}{\mu([\f])}$. 

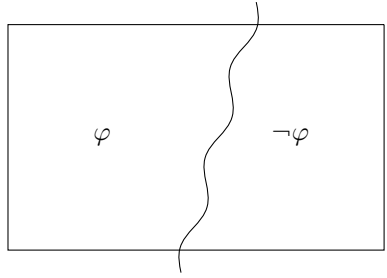
\begin{figure}
\begin{center}

\begin{tikzpicture}
\draw[style={
        decoration={
            complete sines,
            segment length=1cm,
            amplitude=.2cm
        },
        decorate}] (0.8,1.8) -- (-0.2,-1.8);
\draw (2.5,1.5) -- (2.5,-1.5) -- (-2.5,-1.5) -- (-2.5,1.5) -- (2.5,1.5);
\node at (-1.25,0){$\varphi$};
\node at (1.25,0){$\neg\varphi$};
\end{tikzpicture}
\caption{Classic Conditioning}
\label{classicfig}
\end{center}
\end{figure}

Conditionalization in our extended setting follows a similar idea. However, note that both Bayes' and Jeffrey conditioning implicitly rest on the facts that $p(\varphi)+p(\neg\f)=1$ and that $p(\f\wedge\neg\f)=0$, i.e. that there are no gaps and gluts. As this fact no longer holds, conditioning will behave differently in a non-standard setting. In fact, we will show that non-standard probabilities allow for two different notions of Jeffrey updating, one where a new value for the probability of $\f$, i.e. $p(\f)$  is learned, the other where a new value of the four-valued vector $\hat{p}(\f)$ is acquired. The former version of Jeffrey   updating is best described on the level of non-standard probability assignments, the latter on the level of four-valued  assignments. Yet, using the maps $\tr$ and $\trinv$, both versions of updating can naturally be applied to either  non-standard or four-valued probability assignments. 

Just as in the standard case, non-normal Bayes conditioning can be defined as  extremal case of Jeffrey updates. In fact,  non-normal Bayes conditioning has been studied independently, for instance in \cite{Mares}. The current framework generalizes the latter's approach by also incorporating Jeffrey updating and by identifying a number of different Bayes like updates, containing the one put forward by \citeauthor{Mares}.

\subsection{Updating on non-standard information}\label{nsupdating}

In our first notion of updating, the agent's update proscribes her to set the probability of $\f$ to some $q\in[0;1]$. Notably, within a non-standard setting, this  does not carry any information about the value of $\neg\f$ - the agent may or may not leave $p(\neg\f)$ unchanged in her update. In line with classic Jeffrey updating, {\bf non-standard Jeffrey updating} is best illustrated semantically. For any set $\p\in\Lprop$, we can dissect the state space  of a probabilistic model $\mathcal{M} = \langle\Sigma,\mu, v^+, v^-\rangle$ in two sets -- the truth set $[\f]$ of $\f$ and it's complement $\Sigma\setminus[\f]$. Unlike in the classic case, however, $\Sigma\setminus[\f]$ is not the truth set of $[\neg\f]$, nor of any other $\p\in\Lprop$. Yet, we can define Jeffrey updating as in the classic case.

\begin{definition}
Let $\mathcal{M} = \langle\Sigma,\mu, v^+, v^-\rangle$ be a probabilistic model. Let $q\in[0;1]$ and $\f\in\Lprop$ such that $\mu([\f])\in(0;1)$. Then the {\bf semantic non-standard Jeffrey update} for updating the probability of $\f$ to be $q$ on $\mathcal{M}$ is  the probabilistic model $\mathcal{M}^{\f,q} = \langle\Sigma,\mu^{\f,q}, v^+, v^-\rangle$ determined by:
\begin{align*}
\mu^{\f,q}(\{x\})=\begin{cases} \mu(\{x\})\cdot\frac{q}{\mu([\f])} &\text{ iff } x\in[\f]\\ \mu(\{x\})\cdot\frac{1-q}{1-\mu([\f])} &\text{ else. }
\end{cases}
\end{align*}
\end{definition}

\begin{fact}
Non-standard Jeffrey updating is successful, i.e. for any probabilistic model $\mathcal{M} = \langle\Sigma,\mu, v^+, v^-\rangle$, any $q\in[0;1]$ and $\f\in\Lprop$ such that $\mu([\f])\in(0;1)$ the non-standard Jeffrey update on $\M$ updating the probability of $\f$ to $q$  satisfies  $\mu^{\f,q}(x)([\f])=q$.
\end{fact}

Despite the fact that the set $\Sigma\setminus[\f]$ is not definable, we can give a syntactic characterization of non-standard Jeffrey-updating. The following is a non-standard equivalent to classic Jeffrey's updating, cf. Formula (\ref{stJeffrey}).

\begin{lemma}
Let $\mathcal{M} = \langle\Sigma,\mu, v^+, v^-\rangle$ be a probabilistic model. Let $q\in[0;1]$ and $\f\in\Lprop$ such that $\mu([\f])\in(0;1)$. Then for any $\p\in\Lprop$, the non-standard Jeffrey update $\mathcal{M}^{\f,q} = \langle\Sigma,\mu^{\f,q}, v^+, v^-\rangle$ of $\M$ satisfies:
\begin{align*}\mu^{\f,q}([\p])=\mu([\p\wedge\f])\cdot\frac{q}{\mu([\f])}+(\mu(\p)-\mu(\p\wedge\f))\frac{1-q}{1-\mu([\f])}\end{align*}
\end{lemma}

Notably, after translating the previous fact into its induced non-standard probability assignments $p_\mu$ and $p_{\mu^{\f,q}}$, we obtain a fully syntactic characterization of non-standard Jeffrey updating. 

\begin{definition}\label{nsJeffreydef} 
Let $p\colon\Lprop\rightarrow\mathbb{R}$ be a non-standard probability assignment, let $q\in[0;1]$ and $\f\in\Lprop$ with $p(\f)\in(0;1)$. Then the {\bf syntactic non-standard Jeffrey update} setting the probability of $\f$ to $q$ is the probability function $p^{\f,q}\colon\Lprop\rightarrow\mathbb{R}$ defined by
\begin{align*}p^{\f,q}(\p)=p(\p\wedge\f)\cdot\frac{q}{p(\f)}+(p(\p)-p(\p\wedge\f))\frac{1-q}{1-p(\f)}\end{align*}
\end{definition}

\noindent By construction, semantic and syntactic non-standard Jeffrey updating coincide in the following sense. 
\begin{fact}\label{nsJeffreyequiv}
Let $\mathcal{M} = \langle\Sigma,\mu, v^+, v^-\rangle$ be a probabilistic model, let $q\in[0;1]$ and $\f\in\Lprop$ with $p(\f)\in(0;1)$. Then $p_{\mu^{\f,q}}=p_\mu^{\f,q}$.
\end{fact}
 We will hence omit the labels and only speak of non-standard Jeffrey updating. We end this section with three facts about non-standard Jeffrey updating.

\begin{fact}
Assume that the non-standard probability function $p\colon\Lprop\rightarrow\mathbb{R}$ is classic, i.e. satisfies the Kolmogorov axioms. Moreover, let $\f\in\Lprop$ with $p(\f)\in(0;1)$ and $q\in[0;1]$. Then the non-standard and the classic Jeffrey update for setting the probability of $\varphi$ to $q$ coincide, i.e. for all $\p\in\Lprop$
\begin{align*}
p^{\f,q}(\p)=p(\p\wedge\f)\frac{q}{p(\f)}+p(\p\wedge\neg\f)\frac{q}{p(\neg\f)}.
\end{align*}
\end{fact}
\noindent From this, it follows directly that
\begin{fact}
Non-standard Jeffrey updating is not commutative. That is, there is a non-standard probability function $p\colon\Lprop\rightarrow\mathbb{R}$ and $\f,\p\in\Lprop$ and $q,r\in[0;1]$ with $p(\f),p(\p),p^{\f,q}(\p),p^{\p,r}(\f)\in(0;1)$ such that $(p^{\f,q})^{\p,r}\neq (p^{\p,r})^{\f,q}$.
\end{fact}

\subsubsection*{Non-standard Bayesian updating}

Just as in the classic case, we will define non-standard Bayesian updating as  special case of non-standard Jeffrey updating where the probability of $\f$ is set to 1. In this case, the formula of Definition \ref{nsJeffreydef} simplifies to  the same formula as in the classical case. Note that this is also the first of two approaches to Bayes updating proposed by \cite{Mares}. The second proposal by \citeauthor{Mares}, in contrast is not related to any version of Bayes updating presented here, as it strives to actively minimize conflict. 
\begin{definition}
Let $p\colon\Lprop\rightarrow\mathbb{R}$ be a non-standard probability function and let $\f\in\Lprop$ with $p(\f)>0$. Then the {\bf (positive) non-standard Bayesian update} on $\f$ is the function $p ^{\f,pos}$:
\begin{align*}p ^{\f,pos}(\p)=\frac{p(\p\wedge\f)}{p(\f)}\hspace{1cm}\text{ for }\f\in\Lprop.
\end{align*}
\end{definition}

\noindent Unlike in the classical setting, however, non-standard Bayesian updating does not cover all extremal cases. Setting the probability of $\f$ to 0 is {\it not} the same as setting the probability of $\neg\f$ to 1, hence this case needs to be treated separately. 

\begin{definition}
Let $p\colon\Lprop\rightarrow[0;1]$ be a non-standard probability function and let $\f\in\Lprop$ with $p(\f)<1$. Then the {\bf negative non-standard Bayesian update} on $\f$ is the function $p ^{\f,neg}$:
\begin{align*}p ^{\f,neg}(\p)=\frac{p(\p)-p(\p\wedge\f)}{1-p(\f)}\hspace{1cm}\text{ for }\f\in\Lprop.\end{align*}
\end{definition}

\noindent As their classic counterpart, positive and negative non-standard Bayesian conditioning are order independent:

\begin{lemma}
Let $p\colon\Lprop\rightarrow\mathbb{R}$ and let  $\f,\p\in\Lprop$ with $p(\f),p(\p),p^{\f}(\p),p^{\p}(\f)\in(0;1)$. Then $(p^{\f,*})^{\p,\times}=(p^{\p^\times})^{\f,*}$ for $*,\times\in\{pos,neg\}$.
\end{lemma}

\subsection{Updating on four-valued information}\label{fvupdating}

Within non-standard probability, knowing the probability of $\f$ does not provide any information about the probability of $\neg\f$. Hence, in learning about $\f$, two cases are to be distinguished. In the first case, the agent only receives information about $\f$, without learning anything about $\neg\f$ or $\f\wedge\neg\f$. In the second case, the agent learns the full probabilistic information about $\f$, that is, the probabilities of $\f$ and $\neg\f$, but also the size of  the corresponding gap and glut. As discussed above, this information can be encoded in a vector $(b,d,u,c)\in\mathbb{R}^4$ specifying the new pure belief (i.e. belief without conflict), pure disbelief (belief in $\neg\f$ without conflict), uncertainty and conflict about $\f$. 

Again, the notion of four-valued Jeffrey updating is best illustrated semantically. As shown in Figure \ref{fvfig}, for any $\f\in\Lprop$, the sets of pure belief, pure disbelief, uncertainty and conflict about $\f$ jointly form a partition $\left([\f]\setminus[\f\wedge\neg\f],([\neg\f]\setminus[\f\wedge\neg\f],\Sigma\setminus[\f\vee\neg\f],[\f\wedge\neg\f]\right)$ of a probabilistic model $\M$. Hence, a similar idea as in classic Jeffrey updating can be applied, linearly expanding or shrinking the measure on each of these four cells to their appropriate size. Notably, linear expansion (to a larger size) is only well defined if the cell to be expanded has a strictly positive measure. We capture this with the notion of admissibility of a vector $(b,d,u,c)$: 

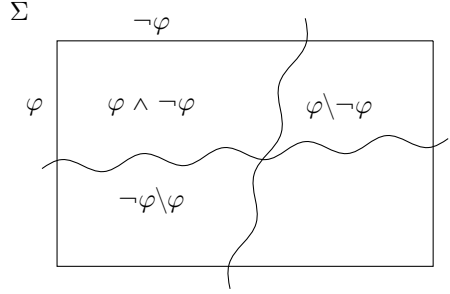
\begin{figure}
\begin{center}

\begin{tikzpicture}
\draw[style={
        decoration={
            complete sines,
            segment length=1cm,
            amplitude=.2cm
        },
        decorate}] (0.8,1.8) -- (-0.2,-1.8);
        \draw[style={
        decoration={
            complete sines,
            segment length=1cm,
            amplitude=.2cm
        },
        decorate}] (-2.7,-0.2) -- (2.7,0.2);

\draw (2.5,1.5) -- (2.5,-1.5) -- (-2.5,-1.5) -- (-2.5,1.5) -- (2.5,1.5);
\node at (-3,1.9){$\Sigma$};
\node at (-2.8,.625){$\varphi$};
\node at (-1.25,1.7){$\neg\varphi$};
\node at (-1.25,.625){$\varphi\wedge\neg\varphi$};
\node at (1.25,.625){$\varphi\setminus\neg\varphi$};
\node at (-1.25,-.625){$\neg\varphi\setminus\varphi$};

\end{tikzpicture}
\caption{Four-valued conditioning}
\label{fvfig}
\end{center}
\end{figure}

\begin{definition}
 Let $\mathcal{M} = \langle\Sigma,\mu, v^+, v^-\rangle$, let $\f\in\Lprop$ and denote $\hat{p}_\mu(\f)$ by $(b_\f,d_\f,u_\f,c_\f)$. We call a vector $(b,d,u,c)\in[0;1]^4$ with $b+d+u+c=1$ {\bf admissible for} $\f$ if it satisfies that $b=0$ if $b_\f=0$, $d=0$ if $d_\f=0$, $u=0$ if $u_\f=0$ and $c=0$ if $c_\f=0$.
\end{definition}
\begin{definition}Let $\mathcal{M} = \langle\Sigma,\mu, v^+, v^-\rangle$ be a probabilistic model, let $\f\in\Lprop$ and let  $(b,d,u,c)\in[0;1]^4$ admissible for $\f$. Then {\bf four-valued Jeffrey updating} on $\f$ to $(b,d,u,c)$ is   the model $\mathcal{M}^{\f,(b,d,u,c)} = \langle\Sigma,\mu^{\f,(b,d,u,c)}, v^+, v^-\rangle$ with:
\begin{align*}
\mu^{\f,q}(x)=\begin{cases} \mu(x)\cdot\frac{b}{\mu([\f]-\mu[\f\wedge\neg\f])} &\text{ iff } x\in[\f]\setminus[\f\wedge\neg\f]\\
 \mu(x)\cdot\frac{d}{\mu([\neg\f])-\mu([\f\wedge\neg\f])} &\text{ iff } x\in[\neg\f]\setminus[\f\wedge\neg\f]\\
 \mu(x)\cdot\frac{c}{\mu([\f\wedge\neg\f])} &\text{ iff } x\in [\f\wedge\neg\f]\\
 \mu(x)\cdot\frac{u}{1-\mu([\f\vee\neg\f])} &\text{ else }
\end{cases}
\end{align*}
\end{definition}

\begin{fact}
Four-valued Jeffrey updating is successful, i.e. for any probabilistic model $\mathcal{M} = \langle\Sigma,\mu, v^+, v^-\rangle$, any $\f\in\Lprop$ and any  $(b,d,u,c)\in[0;1]^4$ that is admissible for $\f$,  the non-standard Jeffrey update on $\M$ setting the probability of $\f$ to $(b,d,u,c)$  satisfies  $\hat{p}_{\mu^{\f,(b,d,u,c)}}(\f)=(b,d,u,c)$.
\end{fact}

Just as in the case of non-standard Jeffrey conditioning, we obtain a purely syntactic characterization of four-valued Jeffrey updating. Unfortunately, the drop in elegance with respect to standard Jeffrey updating is significant. 

\begin{lemma}
Let $\mathcal{M} = \langle\Sigma,\mu, v^+, v^-\rangle$ be a probabilistic model, let $\f\in\Lprop$ and let $(b,d,u,c)\in[0;1]^4$ be admissible for $\f$. Then non-standard Jeffrey update on $\M$ setting the probability of $\f$ to $(b,d,u,c)$ satisfies for any $\p\in\Lprop$ that 
\begin{align*}
b'_\p=&\frac{b}{b_\f}(b_{\f\swedge\p})+\frac{d}{d_\f}(b_{\sneg\f\swedge\p})\\+&\frac{u}{u_\f}(d_{\f\swedge\sneg\f\swedge\p\swedge\sneg\p}-d_{\f\swedge\sneg\f\swedge\p}-c_{\f\swedge\sneg\f\swedge\p}+c_{\f\swedge\sneg\f\swedge\p\swedge\sneg\p})+\frac{c}{c_\f}(c_{\f\swedge\sneg\f\swedge\p}-c_{\f\swedge\sneg\f\swedge\p\swedge\sneg\p})\\
%%%
d'_\p=&\frac{b}{b_\f}(b_{\f\swedge\sneg\p})+\frac{d}{d_\f}(b_{\sneg\f\swedge\sneg\p})\\+&\frac{u}{u_\f}(d_{\f\swedge\sneg\f\swedge\p\swedge\sneg\p}-d_{\f\swedge\sneg\f\swedge\sneg\p}-c_{\f\swedge\sneg\f\swedge\sneg\p}+c_{\f\swedge\sneg\f\swedge\p\swedge\sneg\p})+\frac{c}{c_\f}(c_{\f\swedge\sneg\f\swedge\sneg\p}-c_{\f\swedge\sneg\f\swedge\p\swedge\sneg\p})\\
%%%
u'_\p=&\frac{b}{b_\f}(b_\f-b_{\f\swedge\p}-b_{\f\swedge\sneg\p}-c_{\f\swedge\p}+c_{\f\swedge\sneg\f\swedge\p})+\frac{d}{d_\f}(d_\f-b_{\sneg\f\swedge\p}-b_{\sneg\f\swedge\sneg\p}-c_{\sneg\f\swedge\p}+c_{\f\swedge\sneg\f\swedge\p})\\
&+\frac{u}{u_\f}(1-d_{\f\swedge\sneg\f\swedge\p\swedge\sneg\p}-c_{\f\swedge\sneg\f\swedge\p\swedge\sneg\p})+\frac{c}{c_\f}(c_\f-c_{\f\swedge\sneg\f\swedge\p}-c_{\f\swedge\sneg\f\swedge\sneg\p}+c_{\f\swedge\sneg\f\swedge\p\swedge\sneg\p})\\
%%%
c'_\p=&\frac{b}{b_\f}(c_{\f\swedge\p} - c_{\f\swedge\sneg\f\swedge\p})+ \frac{d}{d_\f}(c_{\sneg\f\swedge \p} - c_{\f\swedge\sneg\f\swedge\p})\\+&\frac{u}{u_\f}(c_\p-c_{\f\swedge\p}+c_{\f\swedge\sneg\f\swedge\p}-c_{\sneg\f\swedge\p}+c_{\f\swedge\sneg\f\swedge\p}-c_{\f\swedge\sneg\f\swedge\p\swedge\sneg\p})+\frac{c}{c_\f}(c_{\f\swedge\sneg\f\swedge\p\swedge\sneg\p})\\
\end{align*}
where $( {b}_\p, {d}_\p, {u}_\p, {c}_\p)$ and $(b'_\p,d'_\p ,u'_\p,c'_\p)$ denote the four-valued probability vector of $\p$ {\it before} and {\it after} the update. In the above equations, $\sneg\p$ is shorthand for $\neg\p$, while $\f\swedge\p$ stands for $\f\wedge\p$.  For ease of notation, this formula uses the convention that $\frac{0}{0}=0$. \end{lemma}
\begin{proof}
Consider the propositions $\f$ and $\p$ as well as the labeling of areas in the top row of Figure \ref{lastsevenpics}. 
By definition of updating, the mass of areas 1-4 need to be multiplied by $\frac{b}{b_\f}$, that of areas  5-8 by $\frac{d}{d_\f}$, the weight of areas  9-12 by $\frac{u}{u_\f}$ and that of areas 13-16 by $\frac{b}{b_\f}$. Moreover, the agent's pure belief in $\p$, i.e. $\overline{b}_\p$ is the joint mass of areas 1, 5, 9 and 13, her disbelief in $\p$ the joint mass of areas 2, 6, 10 and 14, her uncertainty is the joint weight of areas 3, 7, 11 and 15 and her conflict set the sum of areas 4, 8, 12 and 16. 

To check correctness of the above equations, it then suffices to verify that the formulas pick out the respective fields, i.e. that $b_{\f\swedge \p}$ is the size of field 1, $b_{\sneg\f\swedge \p}$ is the size of field 5, $d_{\f\swedge\sneg\f\swedge\p\swedge\sneg\p}-d_{\f\swedge\sneg\f\swedge\p}-c_{\f\swedge\sneg\f\swedge\p}+c_{\f\swedge\sneg\f\swedge\p\swedge\sneg\p}$ is the size of field 9 and so on. That this is the case follows from the pictures in Figure \ref{lastsevenpics}, showing the belief and disbelief sets for certain composites of $\f$ and $\p$.

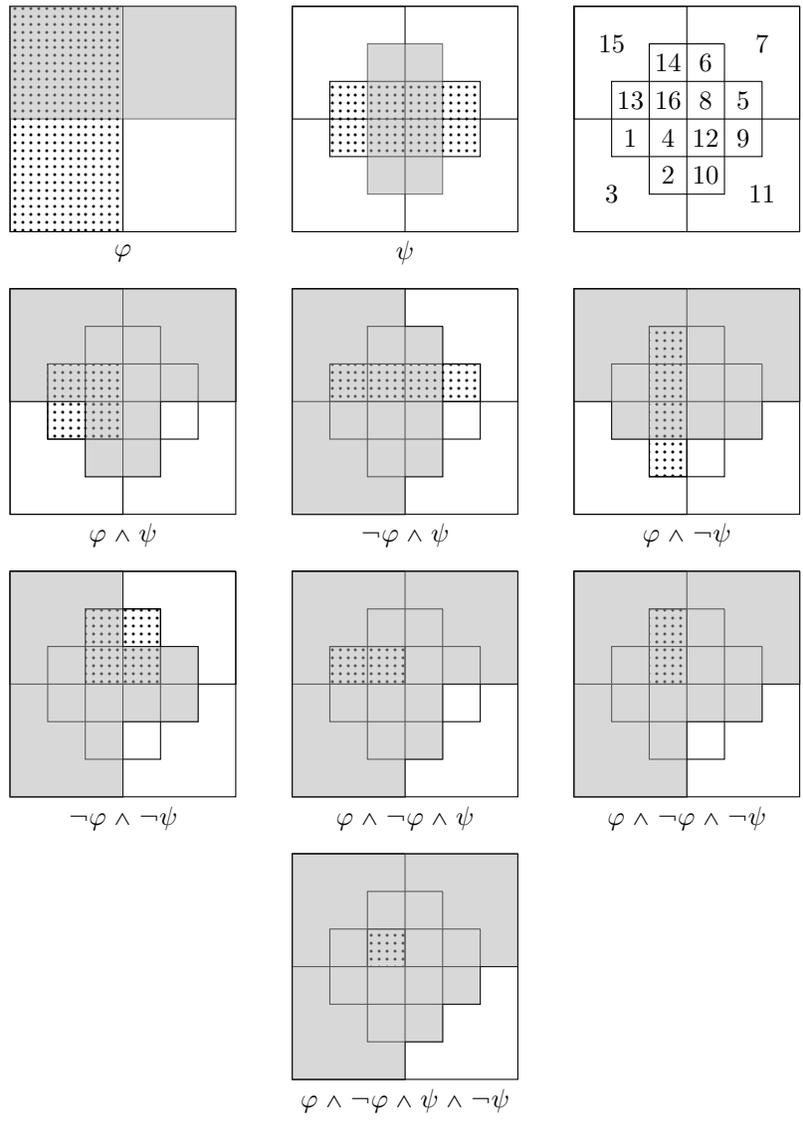
\begin{figure}
\begin{center}
\begin{tikzpicture}
\draw (1.5,1.5) -- (1.5,-1.5) -- (-1.5,-1.5) -- (-1.5,1.5) -- (1.5,1.5);
\draw[pattern=dots, pattern color=black] (0,1.5) -- (0,-1.5) -- (-1.5,-1.5) -- (-1.5,1.5) -- (0,1.5);
\draw[fill=gray!60,opacity=0.5]  (1.5,1.5) -- (1.5,0) -- (-1.5,0) -- (-1.5,1.5) -- (1.5,1.5);
\node at (0,-1.78){$\varphi$};
\node at (0,-2.1){\ };
\end{tikzpicture}
\hspace{.5cm}
\begin{tikzpicture}
\draw (1.5,1.5) -- (1.5,-1.5) -- (-1.5,-1.5) -- (-1.5,1.5) -- (1.5,1.5);
\draw (0,1.5) -- (0,-1.5) -- (-1.5,-1.5) -- (-1.5,1.5) -- (0,1.5);
\draw  (1.5,1.5) -- (1.5,0) -- (-1.5,0) -- (-1.5,1.5) -- (1.5,1.5);
\draw[pattern=dots, pattern color=black] (1,0.5) -- (1,-0.5) -- (-1,-0.5) -- (-1,0.5) -- (1,0.5);
\draw[fill=gray!60,opacity=0.5]  (0.5,1) -- (0.5,-1) -- (-0.5,-1) -- (-0.5,1) -- (0.5,1);
\node at (0,-1.78){$\p$};
\node at (0,-2.1){\ };
\end{tikzpicture}
\hspace{.5cm}
\begin{tikzpicture}
\draw (1.5,1.5) -- (1.5,-1.5) -- (-1.5,-1.5) -- (-1.5,1.5) -- (1.5,1.5);
\draw (0,1.5) -- (0,-1.5) -- (-1.5,-1.5) -- (-1.5,1.5) -- (0,1.5);
\draw  (1.5,1.5) -- (1.5,0) -- (-1.5,0) -- (-1.5,1.5) -- (1.5,1.5);
\draw (1,0.5) -- (1,-0.5) -- (-1,-0.5) -- (-1,0.5) -- (1,0.5);
\draw (0.5,1) -- (0.5,-1) -- (-0.5,-1) -- (-0.5,1) -- (0.5,1);
%\node at (0,-2){$\p$};
\node at (-0.75,-0.25){1};
\node at (-0.25,-0.75){2};
\node at (-1,-1){3};
\node at (-0.25,-0.25){4};
\node at (0.75,0.25){5};
\node at (0.25,0.75){6};
\node at (1,1){7};
\node at (0.25,0.25){8};
\node at (0.75,-0.25){9};
\node at (0.25,-0.75){10};
\node at (1,-1){11};
\node at (0.25,-0.25){12};
\node at (-0.75,0.25){13};
\node at (-0.25,0.75){14};
\node at (-1,1){15};
\node at (-0.25,0.25){16};
\node at (0,-1.78){$ $};
\node at (0,-2.1){\ };
\end{tikzpicture}

\begin{tikzpicture}
\draw (1.5,1.5) -- (1.5,-1.5) -- (-1.5,-1.5) -- (-1.5,1.5) -- (1.5,1.5);
\draw (0,1.5) -- (0,-1.5) -- (-1.5,-1.5) -- (-1.5,1.5) -- (0,1.5);
\draw  (1.5,1.5) -- (1.5,0) -- (-1.5,0) -- (-1.5,1.5) -- (1.5,1.5);
\draw (1,0.5) -- (1,-0.5) -- (-1,-0.5) -- (-1,0.5) -- (1,0.5);
\draw (0.5,1) -- (0.5,-1) -- (-0.5,-1) -- (-0.5,1) -- (0.5,1);
\draw[pattern=dots, pattern color=black] (0,0.5) -- (-1,0.5) -- (-1,-0.5) -- (0,-0.5) -- (0,0.5);
\draw[fill=gray!60,opacity=0.5]  (-1.5,1.5) -- (-1.5,0) -- (-0.5,0) -- (-0.5,-1) -- (0.5,-1) -- (0.5,0) -- (1.5,0) -- (1.5,1.5)--(-1.5,1.5);
\node at (0,-2.1){\ };
\node at (0,-1.78){$\f\wedge\p$};
\end{tikzpicture}
\hspace{.5cm}
\begin{tikzpicture}
\draw (1.5,1.5) -- (1.5,-1.5) -- (-1.5,-1.5) -- (-1.5,1.5) -- (1.5,1.5);
\draw (0,1.5) -- (0,-1.5) -- (-1.5,-1.5) -- (-1.5,1.5) -- (0,1.5);
\draw  (1.5,1.5) -- (1.5,0) -- (-1.5,0) -- (-1.5,1.5) -- (1.5,1.5);
\draw (1,0.5) -- (1,-0.5) -- (-1,-0.5) -- (-1,0.5) -- (1,0.5);
\draw (0.5,1) -- (0.5,-1) -- (-0.5,-1) -- (-0.5,1) -- (0.5,1);
\draw[pattern=dots, pattern color=black] (-1,0.5) -- (-1,0) -- (1,0) -- (1,0.5) -- (-1,0.5);
\draw[fill=gray!60,opacity=0.5]  (-1.5,1.5) -- (-1.5,-1.5) -- (0,-1.5) -- (0,-1) -- (0.5,-1) -- (0.5,1) -- (0,1) -- (0,1.5)--(-1.5,1.5);
\node at (0,-2.1){\ };
\node at (0,-1.78){$\neg\f\wedge\p$};
\end{tikzpicture}
\hspace{.5cm}
\begin{tikzpicture}
\draw (1.5,1.5) -- (1.5,-1.5) -- (-1.5,-1.5) -- (-1.5,1.5) -- (1.5,1.5);
\draw (0,1.5) -- (0,-1.5) -- (-1.5,-1.5) -- (-1.5,1.5) -- (0,1.5);
\draw  (1.5,1.5) -- (1.5,0) -- (-1.5,0) -- (-1.5,1.5) -- (1.5,1.5);
\draw (1,0.5) -- (1,-0.5) -- (-1,-0.5) -- (-1,0.5) -- (1,0.5);
\draw (0.5,1) -- (0.5,-1) -- (-0.5,-1) -- (-0.5,1) -- (0.5,1);
\draw[pattern=dots, pattern color=black] (0,1) -- (-0.5,1) -- (-0.5,-1) -- (0,-1) -- (0,1);
\draw[fill=gray!60,opacity=0.5]  (-1.5,1.5) -- (-1.5,0) -- (-1,0) -- (-1,-0.5) -- (1,-0.5) -- (1,0) -- (1.5,0) -- (1.5,1.5)--(-1.5,1.5);
\node at (0,-2.1){\ };
\node at (0,-1.78){$\f\wedge\neg\p$};
\end{tikzpicture}

\begin{tikzpicture}
\draw (1.5,1.5) -- (1.5,-1.5) -- (-1.5,-1.5) -- (-1.5,1.5) -- (1.5,1.5);
\draw (0,1.5) -- (0,-1.5) -- (-1.5,-1.5) -- (-1.5,1.5) -- (0,1.5);
\draw  (1.5,1.5) -- (1.5,0) -- (-1.5,0) -- (-1.5,1.5) -- (1.5,1.5);
\draw (1,0.5) -- (1,-0.5) -- (-1,-0.5) -- (-1,0.5) -- (1,0.5);
\draw (0.5,1) -- (0.5,-1) -- (-0.5,-1) -- (-0.5,1) -- (0.5,1);
\draw[pattern=dots, pattern color=black] (-0.5,0) -- (0.5,0) -- (0.5,1) -- (-0.5,1) -- (-0.5,0);
\draw[fill=gray!60,opacity=0.5]  (-1.5,1.5) -- (-1.5,-1.5) -- (0,-1.5) -- (0,-0.5) -- (1,-0.5) -- (1,0.5) -- (0,0.5) -- (0,1.5)--(-1.5,1.5);
\node at (0,-2.1){\ };
\node at (0,-1.78){$\neg\f\wedge\neg\p$};
\end{tikzpicture}
\hspace{.5cm}
\begin{tikzpicture}
\draw (1.5,1.5) -- (1.5,-1.5) -- (-1.5,-1.5) -- (-1.5,1.5) -- (1.5,1.5);
\draw (0,1.5) -- (0,-1.5) -- (-1.5,-1.5) -- (-1.5,1.5) -- (0,1.5);
\draw  (1.5,1.5) -- (1.5,0) -- (-1.5,0) -- (-1.5,1.5) -- (1.5,1.5);
\draw (1,0.5) -- (1,-0.5) -- (-1,-0.5) -- (-1,0.5) -- (1,0.5);
\draw (0.5,1) -- (0.5,-1) -- (-0.5,-1) -- (-0.5,1) -- (0.5,1);
\draw[pattern=dots, pattern color=black] (-1,0) -- (0,0) -- (0,0.5) -- (-1,0.5) -- (-1,0);
\draw[fill=gray!60,opacity=0.5]  (-1.5,1.5) -- (-1.5,-1.5) -- (0,-1.5) -- (0,-1) -- (0.5,-1) -- (0.5,0) -- (1.5,0) -- (1.5,1.5)--(-1.5,1.5);
\node at (0,-2.1){\ };
\node at (0,-1.78){$\f\wedge\neg\f\wedge\p$};
\end{tikzpicture}
\hspace{.5cm}
\begin{tikzpicture}
\draw (1.5,1.5) -- (1.5,-1.5) -- (-1.5,-1.5) -- (-1.5,1.5) -- (1.5,1.5);
\draw (0,1.5) -- (0,-1.5) -- (-1.5,-1.5) -- (-1.5,1.5) -- (0,1.5);
\draw  (1.5,1.5) -- (1.5,0) -- (-1.5,0) -- (-1.5,1.5) -- (1.5,1.5);
\draw (1,0.5) -- (1,-0.5) -- (-1,-0.5) -- (-1,0.5) -- (1,0.5);
\draw (0.5,1) -- (0.5,-1) -- (-0.5,-1) -- (-0.5,1) -- (0.5,1);
\draw[pattern=dots, pattern color=black] (-0.5,0) -- (0,0) -- (0,1) -- (-0.5,1) -- (-0.5,0);
\draw[fill=gray!60,opacity=0.5]  (-1.5,1.5) -- (-1.5,-1.5) -- (0,-1.5) -- (0,-0.5) -- (1,-0.5) -- (1,0) -- (1.5,0) -- (1.5,1.5)--(-1.5,1.5);
\node at (0,-2.1){\ };
\node at (0,-1.78){$\f\wedge\neg\f\wedge\neg\p$};
\end{tikzpicture}

\begin{tikzpicture}
\draw (1.5,1.5) -- (1.5,-1.5) -- (-1.5,-1.5) -- (-1.5,1.5) -- (1.5,1.5);
\draw (0,1.5) -- (0,-1.5) -- (-1.5,-1.5) -- (-1.5,1.5) -- (0,1.5);
\draw  (1.5,1.5) -- (1.5,0) -- (-1.5,0) -- (-1.5,1.5) -- (1.5,1.5);
\draw (1,0.5) -- (1,-0.5) -- (-1,-0.5) -- (-1,0.5) -- (1,0.5);
\draw (0.5,1) -- (0.5,-1) -- (-0.5,-1) -- (-0.5,1) -- (0.5,1);
\draw[pattern=dots, pattern color=black] (-0.5,0) -- (0,0) -- (0,0.5) -- (-0.5,0.5) -- (-0.5,0);
\draw[fill=gray!60,opacity=0.5]  (-1.5,1.5) -- (-1.5,-1.5) -- (0,-1.5) -- (0,-1) -- (0.5,-1)-- (0.5,-0.5)-- (1,-0.5) -- (1,0) -- (1.5,0) -- (1.5,1.5)--(-1.5,1.5);
\node at (0,-1.78){$\f\wedge\neg\f\wedge\p\wedge\neg\p$};
\end{tikzpicture}
\end{center}
\caption{Belief and disbelief sets of $\f$ (top left),  $\p$ (top center) and various combinations thereof. Belief sets are dotted, disbelief set shaded. The diagrams fall in 16 sections that are labelled as shown on the top right.}\label{lastsevenpics}
\end{figure}
\end{proof}

Again, the latter set of equations can be read purely syntactically. Thus, we get a syntactic counterpart to semantic four-valued Jeffrey updates.

\begin{definition}
Let $\hat{p}\colon\Lprop\rightarrow\mathbb{R}^4$ be a four-valued probability function and let $\f\in\Lprop$. Moreover, let $(b,d,u,c)\in[0;1]^4$   be admissible for $\varphi$. Then {\bf (syntactic) four-valued Jeffrey updating} with the vector $(b,d,u,c)$ yields a four-valued probability function $\hat{p}^{\f,(b,d,u,c)}$ defined by $\hat{p}^{\f,(b,d,u,c)}(\p)=({b}'_\p,{d}'_\p,{u}'_\p,{c}'_\p)$ with:

\begin{align*}
{b}'_\p=&\frac{b}{b_\f}(b_{\f\swedge\p})+\frac{d}{d_\f}(b_{\sneg\f\swedge\p})\\+&\frac{u}{u_\f}(d_{\f\swedge\sneg\f\swedge\p\swedge\sneg\p}-d_{\f\swedge\sneg\f\swedge\p}-c_{\f\swedge\sneg\f\swedge\p}+c_{\f\swedge\sneg\f\swedge\p\swedge\sneg\p})+\frac{c}{c_\f}(c_{\f\swedge\sneg\f\swedge\p}-c_{\f\swedge\sneg\f\swedge\p\swedge\sneg\p})\\
%%%
{d}'_\p=&\frac{b}{b_\f}(b_{\f\swedge\sneg\p})+\frac{d}{d_\f}(b_{\sneg\f\swedge\sneg\p})\\+&\frac{u}{u_\f}(d_{\f\swedge\sneg\f\swedge\p\swedge\sneg\p}-d_{\f\swedge\sneg\f\swedge\sneg\p}-c_{\f\swedge\sneg\f\swedge\sneg\p}+c_{\f\swedge\sneg\f\swedge\p\swedge\sneg\p})+\frac{c}{c_\f}(c_{\f\swedge\sneg\f\swedge\sneg\p}-c_{\f\swedge\sneg\f\swedge\p\swedge\sneg\p})\\
%%%
{u}'_\p=&\frac{b}{b_\f}(b_\f-b_{\f\swedge\p}-b_{\f\swedge\sneg\p}-c_{\f\swedge\p}+c_{\f\swedge\sneg\f\swedge\p})+\frac{d}{d_\f}(d_\f-b_{\sneg\f\swedge\p}-b_{\sneg\f\swedge\sneg\p}-c_{\sneg\f\swedge\p}+c_{\f\swedge\sneg\f\swedge\p})\\
&+\frac{u}{u_\f}(1-d_{\f\swedge\sneg\f\swedge\p\swedge\sneg\p}-c_{\f\swedge\sneg\f\swedge\p\swedge\sneg\p})+\frac{c}{c_\f}(c_\f-c_{\f\swedge\sneg\f\swedge\p}-c_{\f\swedge\sneg\f\swedge\sneg\p}+c_{\f\swedge\sneg\f\swedge\p\swedge\sneg\p})\\
%%%
{c}'_\p=&\frac{b}{b_\f}(c_{\f\swedge\p} - c_{\f\swedge\sneg\f\swedge\p})+ \frac{d}{d_\f}(c_{\sneg\f\swedge \p} - c_{\f\swedge\sneg\f\swedge\p})\\+&\frac{u}{u_\f}(c_\p-c_{\f\swedge\p}+c_{\f\swedge\sneg\f\swedge\p}-c_{\sneg\f\swedge\p}+c_{\f\swedge\sneg\f\swedge\p}-c_{\f\swedge\sneg\f\swedge\p\swedge\sneg\p})+\frac{c}{c_\f}(c_{\f\swedge\sneg\f\swedge\p\swedge\sneg\p}).
\end{align*}
\end{definition}

By construction, semantic and syntactic non-standard Jeffrey updating coincide in the following sense. 
\begin{fact}\label{fvJeffreyequiv}
Let $\mathcal{M} = \langle\Sigma,\mu, v^+, v^-\rangle$ be a probabilistic model, let $\f\in\Lprop$ and let  $(b,d,u,c)\in[0;1]^4$   be admissible for  $\f$. Then $\hat{p}_{\mu^{\f,(b,d,u,c)}}=\hat{p}_\mu^{\f,(b,d,u,c)}$.
\end{fact}
 We will hence omit the distinction between semantic and syntactic and only speak of four-valued Jeffrey updating. We end this section with three facts about this updating.

\begin{fact}
Assume that the four-valued probability function $\hat{p}\colon\Lprop\rightarrow\mathbb{R}^4$ is classic, i.e. $\hat{p}(\p)\in\mathbb{R}^2\times\{0\}^2$. Moreover, let $\f\in\Lprop$ and $(b,d,0,0)\in[0;1]^4$ be admissible for $\f$, i.e. $b=0$ if $p(\f)=0$ and $d=0$ if $p(\neg\f)=0$. Then the non-standard and the classic Jeffrey update setting the probability of $\varphi$ to $q$ coincide, i.e. for all $\p\in\Lprop$
\begin{align*}
\hat{p}^{\f,(b,d,0,0)}(\psi)=\left(b_{\p\wedge\f}\frac{q}{b_\f}+b_{\p\wedge\neg\f}\frac{q}{b_{\neg\f}},\ 1-b_{\p\wedge\f}\frac{q}{b_\f}-b_{\p\wedge\neg\f}\frac{q}{b_{\neg\f}},\ 0,\ 0\right)
\end{align*}
\end{fact}
\noindent From this, it follows directly that
\begin{fact}
Non-standard Jeffrey updating is not commutative. That is, there is a four-valued  probability function  
 $\hat{p}\colon\Lprop\rightarrow[0;1]^4$, some $\f,\p\in\Lprop$ and $(b,d,u,c),(b',d',u',c')\in[0;1]$ such that $(b,d,u,c)$ is admissible for $\f$ in $\hat{p}$ and  in $\hat{p}^{(b',d',u',c')}$, while $(b',d',u',c')$ is admissible for $\p$ in both $\hat{p}$ and  $\hat{p}^{(b,d,u,c)}$, such that  
 $$(p^{\f,(b,d,u,c)})^{\p,(b',d',u',c')}\neq (p^{\p,(b',d',u',c')})^{\f,(b,d,u,c)}$$
\end{fact}
 
\subsubsection*{Four-valued Bayesian updating}

Just as in the classical case, we can define four-valued Bayesian updating as a special instance of Jeffrey updating where the  information acquired is extremal. Here, we focus on three cases. In the first, the agent learns the vector (1,0,0,0), i.e. she acquires full pure belief in $\varphi$. In the second and third case, the agent learns the vectors (0,0,1,0) or (0,0,0,1) respectively, acquiring full belief in uncertainty or  conflict about $\varphi$. The remaining case, learning (0,1,0,0),  follows from these, as it corresponds to updating on information (1,0,0,0) about  $\neg\f$. In either of our three cases, the above definition of four-valued Jeffrey updating simplifies to:

\begin{definition}
$i)$ Let $\hat{p}\colon\Lprop\rightarrow\mathbb{R}^4$ be a four-valued probability function such that $b_\varphi>0$, where $\hat{p}(\varphi)=(b_\varphi,d_\varphi,u_\varphi,c_\varphi)$. Then {\bf positive four-valued Bayesian updating} on $\varphi$ yields the function $\hat{p}^{\varphi,+}$ defined by $\hat{p}^{\f,+}=(b'_\p,d'_\p,u'_\p,c'_\p)$ with 
\begin{align*}
b'_\p=&\frac{b_{\f\swedge\p}}{b_\f}\\
d'_\p=&\frac{b_{\f,\sneg\p}}{b_\f}\\
u'_\p=&\frac{b_\f-b_{\f\swedge\p}-b_{\f\swedge\sneg\p}-c_{\f\wedge\p}+c_{\f\swedge\sneg\f\swedge\p}}{b_\f}\\
{c}'_\p=&\frac{c_{\f\swedge\p} - c_{\f\swedge\sneg\f\swedge\p}}{b_\f}
\end{align*}

\noindent $ii)$  Let $\hat{p}\colon\Lprop\rightarrow\mathbb{R}^4$ be a four-valued probability function such that $u_\varphi>0$, where $\hat{p}(\varphi)=(b_\varphi,d_\varphi,u_\varphi,c_\varphi)$. Then {\bf uncertainty Bayesian updating} about $\varphi$ is defined as: $\hat{p}^{\f,u}(\p)=(b'_\p,d'_\p,u'_\p,c'_\p)$ with
\begin{align*}
b'_\p=&\frac{d_{\f\swedge\sneg\f\swedge\p\swedge\sneg\p}-d_{\f\swedge\sneg\f\swedge\p}-c_{\f\swedge\sneg\f\swedge\p}+c_{\f\swedge\sneg\f\swedge\p\swedge\sneg\p}}{u_\f}\\
d'_\p=&\frac{d_{\f\swedge\sneg\f\swedge\p\swedge\sneg\p}-d_{\f\swedge\sneg\f\swedge\sneg\p}-c_{\f\swedge\sneg\f\swedge\sneg\p}+c_{\f\swedge\sneg\f\swedge\p\swedge\sneg\p}}{u_\f}\\
u'_\p=&\frac{1-d_{\f\swedge\sneg\f\swedge\p\swedge\sneg\p}-c_{\f\swedge\sneg\f\swedge\p\swedge\sneg\p}}{u_\f}\\
c'_\p=&\frac{c_\p-c_{\f\swedge\p}+c_{\f\swedge\sneg\f\swedge\p}-c_{\sneg\f\swedge\p}+c_{\f\swedge\sneg\f\swedge\p}-c_{\f\swedge\sneg\f\swedge\p\swedge\sneg\p}}{u_\f}
\end{align*}

\noindent $iii)$    Let $\hat{p}:\Lprop\rightarrow\mathbb{R}^4$ be a four-valued probability function such that $c_\varphi>0$, where $\hat{p}(\varphi)=(b_\varphi,d_\varphi,u_\varphi,c_\varphi)$. Then {\bf conflict Bayesian updating} about $\varphi$ is defined as: $\hat{p}^{\f,c}(\p)=({b}'_\p,{d}'_\p,{u}'_\p,{c}'_\p)$ with
\begin{align*}
{b}'_\p=&\frac{c_{\f\swedge\sneg\f\swedge\p}-c_{\f\swedge\sneg\f\swedge\p\swedge\sneg\p}}{c_\f}\\
{d}'_\p=&\frac{c_{\f\swedge\sneg\f\swedge\sneg\p}-c_{\f\swedge\sneg\f\swedge\p\swedge\sneg\p}}{c_\f}\\
{u}'_\p=&\frac{c_\f-c_{\f\swedge\sneg\f\swedge\p}-c_{\f\swedge\sneg\f\swedge\sneg\p}+c_{\f\swedge\sneg\f\swedge\p\swedge\sneg\p}}{c_\f}\\
{c}'_\p=&\frac{c_{\f\swedge\sneg\f\swedge\p\swedge\sneg\p}}{c_\f}
\end{align*}
\end{definition}

\noindent Just as its classic counterpart, four-valued Bayesian conditioning in all three flavors is order independent:

\begin{lemma}
Let $\hat{p}\colon\Lprop\rightarrow[0;1]$ and let  $\f,\p\in\Lprop$ such that $\hat{p}^{\f,a},\ \hat{p}^{\p,b},\ (\hat{p}^{\f,a})^{\p,b}$ and $(\hat{p}_b^{\p,b})^{\f,a}$ are all defined. 
Then $(\hat{p}^{\f,a})^{\p,b}=(\hat{p}_b^{\p,b})^{\f,a}$. for $a,b\in\{{+,u,c}\}$
\end{lemma}

\subsection{Interaction Principles}

Using the translation functions $\tr$ and $\trinv$, both notions of Jeffrey conditioning,  non-standard and four-valued, work on both types of probability functions defined, non-standard and four-valued. However, the notions of updating do not correspond to each other. While non-standard Jeffrey conditioning applies to situations where only the probability of $\f$ is set, without any mention of the probabilities of $\neg\f$ or $\f\wedge\neg\f$, four-valued Jeffrey conditioning covers cases where new probabilities of $\f,\neg\f$ and the corresponding gap and glut are all proscribed simultaneously. Hence, even after appropriate transformations of their domains with $\tr$ and $\trinv$, the two types of Jeffrey updates are {\it not} interdefinable. This, however, changes if we move to non-standard and four-valued Bayesian updating. Each of the three types of four-valued Bayesian updating is equivalent to a composition of two steps of non-standard Bayesian updating. Moreover, the order of these two steps does not matter. 

\begin{lemma}
Let $\hat{p}\colon\Lprop\rightarrow\mathbb{R}^4$ be a four-valued probability assignment and let $\varphi\in\Lprop$.
\begin{itemize}
\item[$i)$] if $b_\varphi>0$, then $\tr(\hat{p}^{\varphi,+})=(\tr(\hat{p})^{\varphi,pos})^{\neg\varphi,neg}=(\tr(\hat{p})^{\neg\varphi,neg})^{\varphi,pos}$
\item[$ii)$] if $u_\varphi>0$, then $\tr(\hat{p}^{\varphi,u})=(\tr(\hat{p})^{\varphi,neg})^{\neg\varphi,neg}=(\tr(\hat{p})^{\neg\varphi,neg})^{\varphi,neg}$
\item[$iii)$] if $c_\varphi>0$, then $\tr(\hat{p}^{\varphi,u})=(\tr(\hat{p})^{\varphi,pos})^{\neg\varphi,pos}=(\tr(\hat{p})^{\neg\varphi,pos})^{\varphi,pos}$
\end{itemize}
\end{lemma}
\begin{proof}
$i)$ By Theorem \ref{repthmfour}, there is a unique canonical model  $\mathcal{M}=\langle\mathcal{P}(\lit), \mu, v^+, v^-\rangle$ such that $\hat{p}_\mu=\hat{p}$. By Facts \ref{nsJeffreyequiv} and \ref{fvJeffreyequiv}, it hence suffices to show the claim for semantic four-valued Jeffrey updating on $\mathcal{M}$. Note that the result of positive Bayesian updating, i.e. the updated four-valued probability function $\mu^{\f,(1,0,0,0)}$ of $\mathcal{M}^{\f,(1,0,0,0)} = \langle\mathcal{P}(\lit),\mu^{\f,(1,0,0,0)}, v^+, v^-\rangle$ is uniquely determined by the conditions
\begin{center}
\begin{tabular}{lll}
(1)& $\mu^{\f,(1,0,0,0)}(x)=0$ &whenever $x\not\in[\varphi]\setminus[\f\wedge\neg\f]$\\
(2)& $\mu^{\f,(1,0,0,0)}(x)/\mu^{\f,(1,0,0,0)}(y)=\mu(x)/\mu(y)$ &whenever $x,y\in[\varphi]\setminus[\f\wedge\neg\f]$\\&& with $\mu(y)>0$.
\end{tabular}
\end{center}
By construction, the two-valued Bayesian updates, i.e. Jeffrey updates $\mu^{\f,1}$ and $\mu^{\neg\f,0}$ both satisfy (2). Moreover, $\mu^{\f,1}(x)=0$ whenever $x\not\in[\f]$ and $\mu^{\neg\f,0}(x)=0$ whenever $x\in[\neg\f]$. Thus both $(\mu^{\f,1})^{\neg\f,0}$ and $(\mu^{\neg\f,0})^{\f,1}$ also satisfy (1). Hence, both  $(\mu^{\f,1})^{\neg\f,0}$ and $(\mu^{\neg\f,0})^{\f,1}$ satisfy conditions (1) and (2) and, hence, are identical to $\mu^{\f,(1,0,0,0)}$. This implies that $\tr(\hat{p}^{\varphi,+})=(\tr(\hat{p}^{\varphi,pos})^{\neg\varphi,neg})=(\tr(\hat{p}^{\neg\varphi,neg})^{\varphi,pos})$.\\
The proofs of $ii)$ and $iii)$ follow similarly. 
\end{proof}

\subsection{Conditioning on Partial Information}

In the previous sections we investigated updating a probability function with a generalized Jeffery rule by learning either only a new value for the  belief  in $\f$ (Section \ref{nsupdating}) or the entire four-valued probability vector assigned to $\f$ (Section \ref{fvupdating}). However, there may be other contexts where the agent   acquires partial   information about the (four-valued) probability of $\f$, e.g. only a new value for pure belief or pure disbelief in $\f$. 

The idea for conditioning on partial information proceeds along the same lines as for complete information, i.e. by a modified version of Jeffery conditioning. The only difference  is that the partiality of information, say about $\f$ does not permit to work with the full partition induced by $\f$ on a model $\mathcal{M}$, i.e. the partition into $\{[\f]\setminus[\neg\f],[\neg\f]\setminus[\f],\Sigma-[\f\vee\neg\f],[\f\wedge\neg\f]\}$, cf Figure \ref{fvfig}, but with a coarsening thereof. 

%yields a coarsing of the partitioning of the underlying $\that is given by $\f$, i.e. $|\f^+| \setminus c_{\f}, |\f^-| \setminus c_{\f}, u_{\f}, c_{\f}$. To this end we allow the agents to update their probability assignment based on partial information as follows. Let $b_{\f}, d_{\f}, u_{\f}, c_{\f} \subseteq \Sigma$ then we denote by $b_{\f}, d_{\f}, u_{\f}, c_{\f}, |\f^+|, |\f^-| = p$ the information that the posterior probability assigned to the {\em pure belief, pure disbelief, uncertaint, conflict} or {\em positive} and {\em negative} support for $\f$. 

By obtaining partial information we mean that the agent learns the values of a {\it partial} assignment $a:\{b,d,u,c\}\rightharpoonup[0;1]$, i.e. an assignment proscribing new values for {\it some} of the agent's pure belief, pure disbelief, uncertainty and conflict, but not necessarily for all. Let us denote the domain of $a$, i.e. those $x\in\{b,d,u,c\}$ for which $a(x)$ is defined, by $dom(a)$. For simplicity, we assume that $\emptyset\subset  dom(a)\subset \{b,d,u,c\}$ with both inclusions strict. Following  the same intuitions as in the four-valued case, we can define conditioning on the partial information $a$ by setting the new pure belief, disbelief, uncertainty and conflict  in $\varphi$ to be $a(b),a(d),a(u)$ and $a(c)$ respectively whenever this is defined and afterwards rescaling the probabilistic mass on the remaining area appropriately. 

Formally, to ensure that the corresponding operation is well-defined, we need to assume that $\sum_{y\in dom(a)}y\leq 1$. Denoting the prior four-valued probability vector of $\f$ with  $(b_\f,d_\f,u_\f,c_\f)$, the Jeffrey updating sketched above will lead to the posterior four-valued probability vector $(\bar{b}_\f,\bar{d}_\f,\bar{u}_\f,\bar{c}_\f)$ with:

\begin{align*}
\bar{x}_\f=\begin{cases}a(x) & \text{iff }x\in dom(a)\\x\cdot \frac{1-\sum_{y\in dom(a)}y}{1-\sum_{y\in dom(a)}y_\f}&else\end{cases}
\end{align*}
for $x\in\{b,d,u,c\}$. With this, we can formally define partial Jeffrey updating. 

\begin{definition}
Let $a:\{b,d,u,c\}\rightharpoonup[0;1]$ be a partial assignment such that $\sum_{y\in dom(a)}y\leq 1$. Let $\mathcal{M} = \langle\Sigma,\mu, v^+, v^-\rangle$ be a model, let $\f\in\Lprop$ and let the vector $(\bar{b}_\f,\bar{d}_\f,\bar{u}_\f,\bar{c}_\f)$ defined above be admissible for $\f$. Then the {\bf four-valued Jeffrey update of  $\varphi$ on the partial information $a$ } is defined as the four-valued Jeffrey update on $\f$ to $(\bar{b}_\f,\bar{d}_\f,\bar{u}_\f,\bar{c}_\f)$.
\end{definition}

\section{Aggregation}\label{aggr}

Assume two agents informed you about their credences in $\varphi$. You take both agents as similarly competent and equally informed. Yet, they equip you with different assessments of $\varphi$. How, then, should you combine these judgments towards forming your own belief about $\varphi$? Within standard probability theory, your options are fairly limited. You may, for instance, decide to follow one of the agents, or build a weighted average between the two. A broad number of approaches in the literature on peer disagreement, for instance, promotes to split the difference equally see for instance \cite{Elga,Christensen} on conciliationism, but also \cite{Kellydisagr} for an opposing opinion.  

\subsection{Aggregating non-standard probabilities}

Within the non-standard probabilities studied here, further options open up. First, note that within classic probability theory, learning about the agents credence in $\p$ also informs us about her degree of belief in $\neg\p$. This does not hold true within the current non-standard setting. Hence, let us assume for the current analysis that  agents inform us about both their positive and negative attitude towards $\varphi$, that is about $p(\varphi)$ and $p(\neg\varphi)$, or even about their four-valued vector $\hat{p}(\varphi)$.  Of course, we may follow the previous strategies and form weighted averages between the agents' assessments of $\varphi$. If needed, this policy could be specified to also taking a weighted average on the agents conflict and uncertainty and, more general, their remaining belief set. 

\begin{definition}
Let $k\in[0,1]$. \\
$i)$ Assume agents $A$ and $E$ provide their non-standard assessments  of $\f$, i.e. $p_A(\varphi),\ p_E(\varphi), p_A(\neg\varphi)$ and $p_E(\neg\varphi)$. Then their {\bf $k$-weighted non-standard aggregate belief} $p^k_{\udb}$ is defined by \begin{align*}p^k_{\udb}(\varphi)=kp_A(\varphi)+(1-k)p_E(\varphi)\text{ and }p^k_{\udb}(\neg\varphi)=kp_A(\neg\varphi)+(1-k)p_E(\neg\varphi).
\end{align*} 
$ii)$ For agent $A$ and $E'$s four-valued probabilitiy assessments  $(b,d,u,c)_A$  and $(b,d,u,c)_E$ for $\varphi$, i.e. $\hat{p}_A(\varphi)$ and $\hat{p}_E(\varphi)$ their {\bf $k$-weighted four-valued aggregate belief} $\hat{p}^k_{\udb}$ is:\begin{align*}\hat{p}_{\udb}^k(\varphi)=k\hat{p}_A(\varphi)+(1-k)\hat{p}_E(\varphi).\end{align*}
\end{definition}

\begin{lemma}\label{univlem} Weighted averaging can be applied to an entire belief base simultaneously. That is, when agents $A$ and $E$ both provide their full subjective non-standard probability functions $p_A,p_E\colon\Lprop\rightarrow\mathbb{R}$ (resp. $\hat{p}_A,\hat{p}_E\colon\Lprop\rightarrow\mathbb{R}^4$), a weighted average belief $p^k_{\udb}\colon\Lprop\rightarrow\mathbb{R}$  can be defined by $kp_A+(1-k)p_E$. Likewise, $\hat{p}^k_{\udb}\colon\Lprop\rightarrow\mathbb{R}$  can be defined by $k\hat{p}_A+(1-k)\hat{p}_E$. Moreover, these policies commute with $\tr$, that is \begin{align*}\tr(\hat{p}^k_{\udb})={p}^k_{\udb}\text{ and }\trinv({p}^k_{\udb})=\hat{p}^k_{\udb}.\end{align*}
\end{lemma}

Non-Standard beliefs, however, allow for further aggregation policies that do not have classic counterparts. {\bf Credulous} agents, for instance, could opt for the maximal values of their input  in terms of  belief and disbelief simultaneously. That is, they could set their updated belief and disbelief in $\varphi$ to be $max(p_{A}(\f), p_{E}(\f))$ and $max(p_{A}(\neg\f), p_{E}(\neg\f))$ respectively. %Denote this values by $p^{cred}_{\udb}(\varphi)$ and  $p^{cred}_{\udb}(\neg\varphi)$. 
Likewise, {\bf cautious} agents may rather chose to belief and disbelief $\varphi$ only to an amount supported by all input information. Such agents would set their %$p^{skept}_{\udb}(\varphi)$ and  $p^{skept}_{\udb}(\neg\varphi)$  
belief and disebelief in $\f$ to  $min(p_{A}(\f), p_{E}(\f))$ and $min(p_{A}(\neg\f), p_{E}(\neg\f))$ respectively.

In special situations, further policies are conceivable. When testing the safety of a new drug, for example, agents may be extremely vary of false positives while being much less concerned with false negatives. Such an agent might decide to set her new belief  in $\varphi$ to $min(p_{A}(\f), p_{E}(\f))$ while adopting $max(p_{A}(\neg\f), p_{E}(\neg\f))$ as new disbelief in $\f$. Likewise, also the combination of $max(p_{A}(\f), p_{E}(\f))$ with $min(p_{E}(\neg\f), p_{E}(\neg\f))$ are conceivable. In some sense, the latter two policies are aggregation functions that minimize type I and type II errors. For a lack of a better name we call these {\bf pessimist} and {\bf optimist} updating rules respectively. See Table \ref{ruletab} for an overview.
\begin{table}
\begin{tabular}{c|cc}
\ \ Updating rule\ \ &$p_{\udb}(\varphi)$&$p_{\udb}(\neg\varphi)$\\[.2em]
\hline
k-weighted&\ \ $kp_A(\varphi)+(1-k)p_E(\varphi)$\ \ &\ \ $kp_A(\neg\varphi)+(1-k)p_E(\neg\varphi)$\ \ \\[.2em]
credulous&$max(p_{A}(\f), p_{E}(\f))$&$max(p_{A}(\neg\f), p_{E}(\neg\f))$\\[.2em]
cautious&$min(p_{A}(\f), p_{E}(\f))$&$min(p_{A}(\neg\f), p_{E}(\neg\f))$\\[.2em]
optimist&$max(p_{A}(\f), p_{E}(\f))$&$min(p_{A}(\neg\f), p_{E}(\neg\f))$\\[.2em]
pessimist&$min(p_{A}(\f), p_{E}(\f))$&$max(p_{A}(\neg\f), p_{E}(\neg\f))$\\[.2em]
\end{tabular}
\caption{Different rules for aggregating agent $A$ and $E$'s non-standard beliefs in $\varphi$ and $\neg\varphi$, i.e. $p_A(\varphi),\ p_E(\varphi), p_A(\neg\varphi)$ and $p_E(\neg\varphi)$.}\label{ruletab}
\end{table}

%Let us denote the aggregates formed by these policies as $p_{\udb}^{fp-min}$ and $p_{\udb}^{tn-min}$ respectively, as they correspond to agents attempting to minimize their false positives resp. true negatives.

\ignore{Notably, none of these four policies uniquely determine the updated four-valued belief  $\hat{p}_{\udb}(\varphi)=(b_\f,d_\f,u_\f,c_\f)$. Given axiom (D2)'s constraint that $b_\f+d_\f+u_\f+c_\f=1$, the four-valued vector $\hat{p}_{\udb}(\varphi)$ is determined by three parameters. All of the above policies, however, only proscribe  two, and hence do not uniquely determine $\hat{p}_{\udb}(\varphi)$. In contrast to the weighted averaging of Policy 1, the missing condition for credulous updates cannot be retrieved by applying the same policy (taking maxima) to the agent's assessments of conflicts about $\varphi$, i.e. to  $p_A(\varphi\wedge\neg\varphi)$ and $p_E(\varphi\wedge\neg\varphi)$. }

Unlike weighted average, none of these four policies can be applied to an entire belief set simultaneously.  

\begin{fact}\label{counterag}
Let $p_A$ and $p_E$ be such that $p_A(\f)=1$ and $p_A(\neg\f)=p_A(\f\wedge\neg\f)=0$, while  $p_E(\neg\f)=1$ and $p_E(\f)=p_E(\f\wedge\neg\f)=0$. Then $p_A$ and $p_E$ are consistent, but the function $p_\udb$  defined by $p_\udb(*)=max(*)$ for $*\in\{\varphi,\neg\varphi,\varphi\wedge\neg\varphi\}$ is not. 
\end{fact}
\begin{proof}
To see that $p_A$ and $p_E$ are consistent consider a nonstandard model with three worlds, $x,y,z$ and $v^+(p)=\{x,y\}, v^-(p)=\{x,z\}$. The measure $\mu_A$ putting all weight on $y$ is such that $p_{\mu_A}(*)=p_A(*)$ for  $*\in\{\varphi,\neg\varphi,\varphi\wedge\neg\varphi\}$, showing $p_A$ consistent by Lemma \ref{sclem}. Likewise, the measure $\mu_E$ putting all weight on $z$ shows $p_E$ consistent. For the inconsistency of $p_\udb$, finally, note that $p_\udb(\f\wedge\neg\f)=0$ and $p_\udb(\f)=p_\udb(\neg\f)=1$. Plugging these three values into (A3) yields $0+p_\udb(\f\vee\neg\f)=2$, contradicting (A1).
\end{proof}

Likewise, the missing conditions for cautious updates cannot be retrieved by extending the policy of taking minima to the agents' assessments of $\varphi\vee\neg\f$, as can be seen from the previous Fact. In particular, there is no counterpart to Lemma \ref{univlem} for credulous or cautious update. Neither can be performed for all $\f\in\Lprop$ simultaneously. 
%new

Before proceeding to four-valued updating, we  compare the above policies to operations in non-probabilistic Belnap-Dunn logic. For this, recall the classic Belnap-Dunn bi-lattice of truth values $\BD$.
%old
\ignore{We are hence in need of a different mechanism for choosing the missing parameter required to define $\hat{p}_{\udb}(\varphi)=(b_\f,d_\f,u_\f,c_\f)$ uniquely. 
Before discussing possible choices for this parameter, however, we compare the above policies to operations in non-probabilistic Belnap-Dunn logic. For this, recall the classic Belnap-Dunn bi-lattice of truth values $\BD$. }

\begin{center}
 \begin{tikzpicture}[>=stealth',relative]

    \node (U1) at (0,-1) {$\{0\}$};
     \node (U2) at (-1,0) {$\{\ \}$};
      \node (U3) at (1,0) {$\{1,0\}$};
       \node (U4) at (0,1) {$\{1\}$};

 \path[-,draw] (U1) to (U2);
 \path[-,draw] (U1) to (U3);
  \path[-,draw] (U2) to (U4);
   \path[-,draw] (U3) to (U4);
    
    \node (L1) at (-0.9,-1.5) {};
     \node (L2) at (0.9,-1.5) {};
      \node (L3) at (-1.5,-0.9) {};
       \node (L4) at (-1.5,0.9) {};
       
       \path[->,draw] (L1) to  node[below, sloped] {information}(L2); 
       \path[->,draw] (L3) to  node[above, sloped] {{truth}}(L4); 
       \end{tikzpicture}
\end{center}

This bi-lattice can be interpreted in two directions relating to  {\it truth values} and the available {\it information}. We denote meet and join of the truth lattice operations by $\wedge$ and $\vee$ while meet and join for the information lattice operations are $\lwedge$ and $\lvee$. Note that we can identify an assignment of $\BD$-values to some formula $\varphi$ with a non-standard probability assignment of $p(\f)$ and $p(\neg\f)$ into $\{0,1\}$. More specifically, assigning $\{1,0\}$ to some $\f$ corresponds to $p(\f)=p(\neg \f)=1$, while assigning $\{1\}$, resp $\{0\}$  to $\varphi$ corresponds to $p(\f)=1,\ p(\neg\f)=0$ and $p(\f)=0,\ p(\neg\f)=1$ respectively. Value $\{\}$, finally, corresponds to $p(\f)=p(\neg \f)=0$. For a probability assignment $p(\f),p(\neg\f)\in\{0,1\}$, we denote the corresponding $\BD$ value by  $t_p(\f)$. Applying this correspondence, we obtain the following characterization of the four updating policies introduced above:

\begin{lemma}
Assume when asked about their credences in $\varphi$, agents $A$ and $E$ provide extremal assignments, i.e. $p_A(\f),p_A(\neg\f),p_E(\f),p_E(\neg\f)\in\{0,1\}$. Then 
\begin{center}
\begin{tabular}{lll}
Credulous update& \multirow{ 2}{*}{\ \ \ yields beliefs in $\varphi$ and $\neg\varphi$ \ \ \ }& $t_{p_A}(\f)\lvee t_{p_E}(\f)$\\
Cautious update&\multirow{ 2}{*}{\ \ \ that are equal to}&$t_{p_A}(\f)\lwedge t_{p_E}(\f)$\\
Opimistic update&&$t_{p_A}(\f)\vee t_{p_E}(\f)$\\
Pessimistic update&&$t_{p_A}(\f)\wedge t_{p_E}(\f)$.

\end{tabular}
\end{center}
\end{lemma}

Finally, we consider the special case where both agents input {\it classic} probability values, i.e. values such that $p(\f)+p(\neg\f)=1$.\\

\begin{fact}When $p_A$ and $p_E$ are classic, i.e. $p_A(\f)+p_A(\neg\f)=p_E(\f)+p_E(\neg\f)=1$, then the same holds for the aggregated belief when aggregation follows weighted averaging, optimistic or pessimistic updates. That is, these three rules preserve classicality. This does not hold for credulous and cautious updating. The latter two rules turn classic inputs beliefs for agent $A$ and $E$ into non-classic aggregate values as soon as $A$ and $E$ disagree about $p(\f)$. 
\end{fact}

%%%%%%

\subsection{Aggregating four-valued probabilities.}
So far, we have assumed aggregation to operate on non-standard probability assignments. Within the above framework, agents provide  their subjective non-standard beliefs in both $\varphi$ and $\neg\f$,  which the various aggregative mechanisms described above then merge into aggregate belief values for $\varphi$ and $\neg\varphi$. But of course, our agents might also  provide their subjective four-valued probabilities $\hat{p}_A(\f)=(b^A_\f,d^A_\f,u^A_\f,c^A_\f)$ and $\hat{p}_E(\f)=(b^E_\f,d^E_\f,u^E_\f,c^AE_\f)$ instead. Naturally, we could then hope to obtain an aggregate four-valued probability \begin{align*}\hat{p}_\udb(\f)=(b^\udb_\f,d^\udb_\f,u^\udb_\f,c^\udb_\f).\end{align*}

Note, that by the map $\tr$, the non-standard probabilities $p(\f)$ and $p(\neg\f)$ can be calculated from the four-valued probability $\hat{p}(\f)$. Hence, if $\hat{p}_\udb(\f)$ is defined, a corresponding two-valued aggregation mechanism for $p_\udb(\f)$ and $p_\udb(\neg\f)$ follows immediately. However, the opposite does not hold. $p_\udb(\f)$ and $p_\udb(\neg\f)$ do not fully determine $\hat{p}_\udb(\f)$ and hence the various policies defined in the last section do not readily translate into four-valued aggregation procedures. 
In fact, when employing the map $\trinv$, the three values $p(\f)$,  $p(\neg\f)$ and $p(\f\wedge\neg\f)$ are required to determine $\hat{p}(\f)$. In the case of weighted averaging, this is not a problem. By Lemma \ref{univlem}, setting \begin{align*}p^k_\udb(\f\wedge\neg\f)=kp^k_A(\f\wedge\neg\f)+(1-k)p^k_E(\f\wedge\neg\f)\end{align*} yields a consistent set of requirements and the corresponding four-valued aggregation rule is exactly $\hat{p}^k_\udb(\f)=k\hat{p}_A(\f)+(1-k)\hat{p}_E(\f)$.

However, the situation is different in the case of credulous or cautious updating. As shown in Fact \ref{counterag}, requiring that $p_\udb(\f)=max(p_A(\f),p_E(\f))$, $p_\udb(\neg\f)=max(p_A(\neg\f),p_E(\neg\f))$ and  $p_\udb(\f\wedge\neg\f)=max(p_A(\f\wedge\neg\f),p_E(\f\wedge\neg\f))$ may yield an {\it inconsistent} set of requirements. Hence, other choices are needed. 

The  vector $p_\udb(\f)$ is determined by four choices. With two of them given by $p_\udb(\f)=max(p_A(\f),p_E(\f))$ and $p_\udb(\neg\f)=max(p_A(\neg\f),p_E(\neg\f))$, and a third by axiom (D2), one last condition is missing. In the case of credulous update, we would arguably expect that $c_\f^\udb\geq max(c^A_\f,c^E_\f)$: If an agent   opts to be credulous about  both $\varphi$ and $\neg\varphi$, she could not expect her conflict to fall below any of the input conflicts. Within this restriction, the below definition of credulous update,  assumes $c_\f^\udb$ to be as close to  $max(c^A_\f,c^E_\f)$ as possible while maintaining consistency. 

Likewise, in the case of cautious update, we would arguably expect overall uncertainty to grow, or, at least, not to shrink through aggregation. That is, we would expect that $u_\f^\udb\geq max(u^A_\f,u^E_\f)$. Again, We will demand that $u_\f^\udb$ is the maximal possible consistent value with this property.

\begin{definition}
Assume agents $A$ and $E$ provide four-valued probabilities $\hat{p}_A(\f)=(b^A_\f,d^A_\f,u^A_\f,c^A_\f)$ and $\hat{p}_E(\f)=(b^E_\f,d^E_\f,u^E_\f,c^AE_\f)$. Then the {\bf credulously aggregated four-valued probability} $\hat{p}_\udb(\f)=(b^\udb_\f,d^\udb_\f,u^\udb_\f,c^\udb_\f)$  is given by the following four conditions 
\begin{center}
\begin{tabular}{l}
$b^\udb_\f+c_\f^\udb=max(b_\f^A+c_\f^A,b_\f^E+c_\f^E)$  \\
$d^\udb_\f+c_\f^\udb=max(d_\f^A+c_\f^A,d_\f^E+c_\f^E)$  \\
$b^\udb_\f+d^\udb_\f+u^\udb_\f+c^\udb_\f=1$\\
$c^\udb_\f=max\left(c^E_\f,\ c^A_\f,\ (b^\udb_\f+c_\f^\udb)+(d^\udb_\f+c_\f^\udb)-1\right)$
\end{tabular}
\end{center}
\end{definition}
By $\tr$, the first two of these equations correspond to the two conditions of credulous non-standard updates, i.e.\ $p_\udb(\varphi)=max(p_A(\varphi), p_E(\varphi))$ and $p_\udb(\neg\varphi)=max(p_A(\neg\varphi),p_E(\neg\varphi))$. The third equation is axiom (D2). The last equation, finally expresses that $c^\udb_\f$ is the minimal consistent choice such that $c_\f^\udb\geq max(c^A_\f,c^E_\f)$. To see this, note that by (D2), we have $b^\udb_\f+d^\udb_\f+c_\f^\udb\leq 1$ and hence \begin{align*}b^\udb_\f+c_\f^\udb+d^\udb_\f+c_\f^\udb-1\leq c_\f^\udb.\end{align*}

\noindent Likewise we can define a cautious aggregation of four-valued probabilities:

\begin{definition}
For $\hat{p}_A$ and $\hat{p}_E$ as above, the {\bf cautiously aggregated four-valued probability} $\hat{p}_\udb(\f)=(b^\udb_\f,d^\udb_\f,u^\udb_\f,c^\udb_\f)$ is given by  the following four equations
\begin{center}
\begin{tabular}{l}
$b^\udb_\f+c_\f^\udb=min(b_\f^A+c_\f^A,b_\f^E+c_\f^E)$  \\
$d^\udb_\f+c_\f^\udb=min(d_\f^A+c_\f^A,d_\f^E+c_\f^E)$  \\
$b^\udb_\f+d^\udb_\f+u^\udb_\f+c^\udb_\f=1$\\
$u^\udb_\f=max\left(u^E_\f,\ u^A_\f,\ 1-(b^\udb_\f+c_\f^\udb)-(d^\udb_\f+c_\f^\udb)\right)$
\end{tabular}
\end{center}
\end{definition}

 \noindent Credulous and cautious aggregation as defined here cohere with their definition for non-standard probabilities.

\begin{lemma}
Assume that agents $A$ and $E$ provide four-valued vectors $\hat{p}_A$ and $\hat{p}_E$ respectively. Then the following diagrams commute, where the application of $\tr$ makes use of the fact that $p(\varphi)$ and $p(\neg\f)$  can be calculated from $\hat{p}(\f)$.

\begin{center}
 \begin{tikzpicture}[>=stealth',relative]

    \node (U1) at (0,3) {\begin{tabular}{l}$\hat{p}_A(\varphi)$\\$\hat{p}_E(\varphi)$\end{tabular}};
     \node (U2) at (0,0) {$\hat{p}_\udb(\varphi)$};
      \node (U3) at (4,3) {\begin{tabular}{l}	$p_A(\varphi),p_A(\neg\f)$\\$p_E(\f),p_E(\neg\f)$\end{tabular}};
       \node (U4) at (4,0) {$p_\udb(\varphi)$};

 \path[->,draw] (U1) to  node[below, sloped] {$\tr$}(U3);
  \path[->,draw] (U1) to  node[below, sloped] {credulous} (U2);
  %Alternative code if one wants to specify the difference between non-standard and four-valued
% \path[->,draw] (U1) to  node[below, sloped,align=center] {(four-valued)\\credulous} (U2);
  \path[->,draw] (U3) to  node[below, sloped] {credulous}(U4);
   \path[->,draw] (U2) to  node[below, sloped] {$\tr$}(U4);
    
    \node (L1) at (7,3) {\begin{tabular}{l}$\hat{p}_A(\varphi)$\\$\hat{p}_E(\varphi)$\end{tabular}};
     \node (L2) at (7,0) {$\hat{p}_\udb(\varphi)$};
      \node (L3) at (11,3) {\begin{tabular}{l}	$p_A(\varphi),p_A(\neg\f)$\\$p_E(\f),p_E(\neg\f)$\end{tabular}};
       \node (L4) at (11,0) {$p_\udb(\varphi)$};

 \path[->,draw] (L1) to  node[below, sloped] {$\tr$}(L3);
 \path[->,draw] (L1) to  node[below, sloped] {cautious} (L2);
  \path[->,draw] (L3) to  node[below, sloped] {cautious}(L4);
   \path[->,draw] (L2) to  node[below, sloped] {$\tr$}(L4);
       \end{tikzpicture}

\end{center}
\end{lemma}

\subsubsection*{The algebraic structure of credulous and cautious aggregation. }
\begin{definition}
For an aggregation strategy $S$, we call $\f$ a neutral element if for all $\p$ we have  
$$S(\f, \p) =S(\p, \f)= \p,$$
and we call $\f$ an anihilator if for all $\p $
$$S(\f, \p) =S(\p, \f)= \f.$$
\end{definition}

\begin{proposition}
The subjective four-valued probability assignment $(0,0,0,1)$, i.e. the element of maximal conflict, is an anihilator with respect to credulous updating.  Likewise, the  subjective four-valued probability assignment $(0,0,1,0)$, representing maximal uncertainty, is an anihilator for the cautious strategy.
\end{proposition}

\begin{proof}
Let   $\hat{p}_A(\f)=(0,0, 0,1)$, let $\hat{p}_E(\f)$ be arbitrary and denote the result of credulous updating by  $(B, D, U, C)$. Then by definition $B+C=D+C=1$ which, together with (D1) and (D2) implies $C=1$, $B=D=0$, and hence $U=0$. In a similar manner, let  $\hat{p}_A(\p)=(0,0, 1,0)$ and let  $\hat{p}_E(\p)$ be arbitrary, and denote the result of cautious updating by $(B', D', C', U')$.  then by definition $B'+C'= D'+C'=0$, which implies $B'=D'=C'= 0$ and hence $U'=1$.

\end{proof}

\begin{proposition}
The subjective four-valued probability assignment $(0,0,0,1)$, i.e. the element of maximal conflict, is a neutral element with respect to cautious updating.  Likewise, the  subjective four-valued probability assignment $(0,0,1,0)$, representing maximal uncertainty, is neutral with respect to credulous updating.\end{proposition}

\begin{proof}
Let   $\hat{p}_A(\f)=(0,0, 0,1)$, let $\hat{p}_E(\f)=(b,d, u,c)$ be arbitrary and denote the result of cautious updating by  $(B, D, U, C)$. Then by definition $B+C=b+c$ and $D+C=d+c$ which implies 
\begin{equation}\label{eq1}
B+D+2C= b+d + 2c.
\end{equation} 
Using this, the last condition of cautious updating yields $U=max(0,\ u,\ 1-b-d-2c)$. Since $u=1-b-d-c$, this implies $U=u$. Together with $1=U+B+C+D=u+b+c+d$, it follows that $B+C+D=b+c+d$. In combination with equation (\ref{eq1}), this implies $C=c$. With this, $B+C=b+c$ and $D+C=d+c$ imply that $B=b$ and $D=d$. The proof for the second claim follows from a similar argument. 

\end{proof}

\section{Conclusions}\label{concl}

Many classical approaches to  reasoning address idealized situations, where the agents' information is consistent, closed under  logical implication, and possibly even complete. These assumptions, of course, are at odds with many realistic reasoning scenarios, where the available evidence may be scarce and memory or observation faulty. 
In short, there is no guarantee for our available information to be consistent, nor complete. Yet, we would arguably hold that {\it some} valid inferences can be drawn from such imperfect information, as partial incompleteness or local contradictions may not preclude us from drawing  conclusions about  other parts of the data.  
As automated reasoning systems are becoming increasingly important, there is a need for a rigorous formal treatment of inferences from non-ideal information.
To this end, a wealth of  non-classical logical systems for dealing with uncertainty or conflict has been put forward, with Belnap-Dunn logic (BD) arguably the most prominent such framework. 

However, the reasons for moving to non-normal, BD like frameworks apply equally well to probabilistic settings. Agents may, for instance, have inconclusive, probabilistic evidence for the truth or falsity of various statements. Just as in the classic case, if such information comes from different sources or different experiments, it needs not add up to 1, nor be mutually exclusive. It hence seems natural to investigate probabilistic extensions of BD.  This was the focus of the current paper. 
%%%

%%%
%%%
Paralleling recent work by \citeauthor{DunnProb} \citep[cf.][]{DunnProb,dunn2019contradictory}, we have investigated four-valued probability assignments that permit agents to have probabilistic beliefs about the truth and falsity of a statement, and about its gaps and gluts. More specifically, we have provided a theory of four-valued probabilities that slightly departs from Dunn's in its treatment of conjunctions. Yet, both  are generalizations of Belnap-Dunn logic in that they coincide with BD whenever all probabilities are extremal, i.e. only assume the values of 0 and 1. 

In this paper, we have clarified the connection between our four-valued probabilities and  single valued non-standard probabilities as introduced by Childers, Majer and Milne (\citeyear{childers_at_al}). By providing a translation function between the two approaches, we have shown these to be equivalent. Moreover, we have introduced probabilistic models  as semantics for four-valued probabilities, and have provided a sound and complete axiomatization  with respect to the class of all such models.   Lastly, we have enriched our frameworks with dynamical operations for updating and aggregation. As for the former, we have provided versions of Jeffrey and Bayes' conditioning that work in non-standard and four-valued settings and have clarified the relation between these.  For aggregation, finally,  we have studied a host of different aggregation policies, some of which go beyond what is available in classic probabilistic settings.

%To equip the probabilities defined with a semantics, we have introduced probabilistic models and provided a sound and complete axiomatization of  four-valued probability assignments with respect to these models. Lastly, we have enriched our frameworks with dynamical operations for updating and aggregation. As for the former, we have provided versions of Jeffrey and Bayes' conditioning that work in non-standard and four-valued settings and have clarified the relation between these.  For aggregation, finally,  we have studied a host of different aggregation policies, some of which go beyond what is available in classic probabilistic settings. 

%%%

Of course, there are other approaches to weakening classic probability theory, not all of which have a corresponding logic as starting point. Many such approaches take probability or weights as central notion, but consider various cases where no exact probabilistic information is available. A typical example are  inner measures  intended to approximate probability from below \citep{faginhalpern}. Their underlying idea, briefly, is that an agent might lack  probabilistic evidence about some proposition $\f$, for instance when $\f$ is not in the algebra of (possible) observations. 
%In such situations, the agent can employ whatever probabilistic information she has to estimate a lower bound for the probability of $\f$.  
%In such situations, the agent may  estimate a lower bound for the probability of $\f$ by employing her available probabilistic information about other propositions. 
The agent may, though,  estimate a lower bound for the probability of $\f$ by building on her available  information about other propositions. 
Formally, this gives rise to an inner measures that  only satisfy  super-additivity instead of the classic additivity, i.e.  $\mu_*(\f \lor \p)\geq \mu_*(\f) + \mu_*(\p) $, where $\f \land \p$ a classical contradiction.

 %belief in a proposition $A$ is a normalized measure on a boolean algebra expressing the amount of evidence in support of the hypothesis A

 A related weakening of classic probability theory is   Dempster-Shafer (DS) theory of belief  \citep{shafer,halpern2017}. 
The starting point of this theory is an agent's  evidence about some state of affairs, usually represented as a normalized measure on a boolean algebra of possible observations.  This evidence then gives rise to a belief function, where $Bel(\f)$, the belief in some  $\f$, is derived from all pieces of evidence that entail $\f$. As the agent might have strong evidence for a compound event, say $\p\vee \f$, without having much evidence that entails either of its compounds alone, this belief function is super-additive in the sense defined above. More specifically, the  degree of support for some $A$ needs not be complementary to the support of $\neg A$. That is,  $Bel(A)$ may be less than $1-Bel (\neg A)$,  just as in our framework. While $Bel(A)$ can be seen as  a lower bound for the classical probability for $A$, the term $1-Bel(\neg A)$, sometimes denoted the plausibility of $A$, is it's upper bound. The interval between both is then interpreted as the agent's uncertainty about $A$. As our presentation suggests, there is a tight connection between DS theory and inner measures approaches: both are equivalent, at least on a syntactic level where probabilities are associated to formulas, rather than states  \citep{faginhalpern,Zhou}.  
 
Both, inner probabilitiy approaches and DS theory differ in two ways from our framework. In one dimension, our framework is more general than DS belief functions or inner probabilities, as it admits not only for uncertainty but also for conflict in probability assignments. By allowing for gluts, non-standard and four-valued probability assignments can represent contradictory information in ways that DS theory and inner measure frameworks cannot. 

For a second difference consider a classic tautology such as $p\vee \neg p$. Working on a classical meta-theory, DS theory associates a probability of 1 to this tautology. Yet, when evidence is scarce, the belief values assigned to $p$ and $\neg p$ need not add up to one, exemplifying the above super-additivity. In fact, it is compatibly with DS theory that both $p$ and $\neg p$ are even assigned a belief of zero.  In our framework, in contrast, uncertainty or conflict derive straight from the information available about $p$ and $\neg p$, rather than from evidence about some larger proposition. Working with an non-classic, BD-metatheory, non-classic information about literals extends to  complex formulas such as   $ p\vee\neg p$, as witnessed in the inclusion-exclusion axiom (A3). This axiom, in fact, can be seen to stand in direct opposition to the theory of inner measures. Our axioms (A3) implies a subadditivity property  (i.e. $\mu_*(\f \lor \p)\leq \mu_*(\f) + \mu_*(\p) $ when $\f \land \p$ is a  classical contradiction), in contrast  to the superadditivity of DS theory and inner measures.   A detailed comparison beyond DS belief functions and our approach would require a more careful analysis that exceeds the scope of this article. We leave this for future work.

%The Dempster Shafer approach  is similar to ours in treating the support of a hypothesis independently of the support of its negation and allowing a representation of uncertainty. Our approach, however,  seems to be more general as it also admits not only uncertainty but also conflict in probability assignments. On the other hand the DS theory seems to be more general approach in the sense that it relaxed the constraints on the probability assignments that are imposed by the logical dependence between propositions. For instance the DS theory allows for assignment of belief 1 to a disjunction $p \vee q$ while assigning belief zero to both $p$ and $q$ which is impossible in our setting.  A detailed comparison would require a more careful analysis and goes beyond the scope of the article and we leave for future work.

Finally, another open line of inquiry concerns practical implications of the present framework. One may, for instance, ask how an ideally rational agent is to act if she has only imperfect information at her disposal. In future work, we hope to sketch the contours of a non-standard decision theory,  that rests on four-valued probabilities in the same manner as traditional decision theory employs classic probability.  Doing so, we hope, can help to fill a gap between current frameworks for decisions under risk and  under uncertainty.

%On the other hand Formally Bel satisfies the following axioms:… (in the terms of term algebra) (usually given by mass function, direct axiomatization e.g. Halpern)
%The third axiom seems to be a weaker version of our inclusion/exclusion with inequality, 

%%%%%%%%%%%%%%%%%%%%%%%%%%%%%%%%%%%%%%%%%%%%%%%%
\nocite{PriestBook}
\bibliographystyle{chicago}

	\end{document}